\documentclass[11pt,reqno]{amsart}
\usepackage{a4wide,fullpage,verbatim,amssymb}
\usepackage{bbm}
\usepackage{bm}
\usepackage{tikz-cd}
\usepackage[normalem]{ulem}
\usepackage{amsmath}
\usepackage{soul, mathtools}
\usepackage[active]{srcltx}
\usepackage{xcolor}
\usepackage{amsmath}
\usepackage{enumitem}
\usepackage{cleveref}

\makeatletter
\let\@@pmod\bmod
\DeclareRobustCommand{\bmod}{\@ifstar\@pmods\@@pmod}
\def\@pmods#1{\mkern4mu({\operator@font mod}\mkern 6mu#1)}
\makeatother

\definecolor{blue}{rgb}{0,0,1}
\definecolor{red}{rgb}{1,0,0}
\definecolor{green}{rgb}{0,.6,.2}
\definecolor{purple}{rgb}{1,0,1}

\long\def\red#1\endred{\textcolor{red}{#1}}
\long\def\blue#1\endblue{\textcolor{blue}{#1}}
\long\def\purple#1\endpurple{\textcolor{purple}{ #1}}
\long\def\green#1\endgreen{\textcolor{green}{#1}}

\numberwithin{equation}{section}

\newcommand{\g}{\gamma}
\newcommand{\G}{\Gamma}

\newcommand{\im}{\text{Im}}
\newcommand{\br}{\text{\bf r}}
\newcommand{\De}{D^{\omega, \text{exc}}}
\newcommand{\lp}{\left (}
\newcommand{\rp}{\right )}

\newcommand{\Z}{\mathbb{Z}}
\newcommand{\N}{\mathbb{N}}
\newcommand{\Q}{\mathbb{Q}}

\newcommand{\C}{\mathbb{C}}

\newcommand{\stkout}[1]{\ifmmode\text{\red \sout{\ensuremath{#1}} \endred}\else\sout{#1}\fi}

\DeclareMathOperator{\SL}{SL}

\newcommand{\DD}{D^{\omega, \infty, \text{exc}}}

\allowdisplaybreaks

\def\ca{{\mathfrak a}}

\crefname{subsection}{subsection}{Subsection}

\DeclareMathOperator{\sgn}{{\rm sgn}} 

\newcommand\hol{{\mathcal{O}}}
\newcommand\indlim{{\displaystyle\lim_{\longrightarrow}\,}}

\newcommand{\sm}{\left(\begin{smallmatrix}}
\newcommand{\esm}{\end{smallmatrix}\right)}
\newcommand{\bpm}{\begin{pmatrix}}
\newcommand{\ebpm}{\end{pmatrix}}

\newtheorem{theorem}{Theorem}[section]
\newtheorem{lemma}[theorem]{Lemma}
\newtheorem{proposition}[theorem]{Proposition}
\newtheorem{corollary}[theorem]{Corollary}
\newtheorem{definition}[theorem]{Definition}

\newtheorem*{remarks}{Remarks}

\newtheorem*{examples}{Examples}

\title{Iterated integrals and cohomology}
\author[K. Bringmann]{Kathrin Bringmann}
\address{Department of Mathematics and Computer Science\\Division of Mathematics\\University of Cologne\\ Weyertal 86-90 \\ 50931 Cologne \\Germany}
\email{kbringma@math.uni-koeln.de}
\author{Nikolaos Diamantis}
\address{University of Nottingham}\email{nikolaos.diamantis@nottingham.ac.uk}

\begin{document}
\begin{abstract}
We introduce an extension of the standard cohomology which is characterised by maps that fail to be classical cocycles by products of simpler maps. The construction is motivated by the study of Manin's noncommutative modular symbols and of false theta functions. We use this construction to obtain a cohomological interpretation of important iterated integrals that arise in that study. In another direction, our approach gives modular counterparts to the long-studied relations among multiple zeta values. 
\end{abstract}
\maketitle
\vskip -8mm
\section{Introduction and statement of results}
Iterated integrals increasingly attract attention in various areas, including number theory and physics. This paper is motivated by their occurrences as Manin's iterated Mellin transforms and as certain iterated integrals that appeared in the study of false theta functions \cite{BKM,BN}. In the introduction, we consider a more restricted case to avoid technicalities. In the main body of the paper, we then deal with iterated integrals of general length, which we call \emph{depth} (Section \ref{higher}), and with both integral and half-integral weights. 

We start with integral weight. To be more precise, let $f_1, f_2\in S_k(\Gamma_0(N))$, the space of cusp forms of weight $k\in2\N$ for $\Gamma_0(N)$. For  $\gamma \in \Gamma_0(N)$, we define\footnote{This map previously appeared in the literature in other guises.} the \emph{period integral of depth $2$} by\footnote{This integral is initially only defined for $\tau$ in the lower half-plane $\mathbb H^-$ and in $\mathbb Q \cup \{i \infty\}$.}
\begin{equation*}
	r_{f_1, f_2}(\gamma)(\tau):=\int_{\gamma^{-1}i\infty}^{i \infty} f_1(w_1)(w_1-\tau)^{k_1-2}\int_{w_1}^{i \infty} f_2(w_2)(w_2-\tau)^{k_2-2}dw_2 dw_1,
\end{equation*}
The coefficients of this integral may be expressed in terms of the {\it multiple $L$-series of depth $2$}
$$
	L_{f_1, f_2}\lp \frac ab; s_1, s_2 \rp:=\sum_{n_1, n_2 \ge 1}\frac{c_1(n_1) c_2(n_2)e^{2 \pi i (n_1+n_2)\frac ab}}{(n_1+n_2)^{s_1}n_2^{s_2}},
$$
where $\operatorname{Re}(s_1), \operatorname{Re}(s_2) \gg 1$, $\frac ab\in\mathbb Q$, and $c_j(n)$ denotes the $n$-th Fourier coefficient of $f_j$ (see Proposition \ref{MellLser}). We also require its ``completed'' version
$$\Lambda_{f_1, f_2}\lp \frac ab; s_1, s_2 \rp:=
\int_{\frac ab}^{i \infty} f_1(w_1) \lp w_1-\frac ab \rp ^{s_1-1} \int_{w_1}^{i \infty} f_2(w_2)(w_2-w_{1})^{s_2-1}dw_2dw_1.$$

We next turn to the case of half-integral weight. For $j \in \{1, 2\}$ let $\chi_j$ be a multiplier system of weight $k_j \in\mathbb Z+\frac12$ for $\Gamma_0(N)$ with $4|N$, and $f_j \in S_{k_j}(\Gamma_0(N), \chi_j)$ where $S_{k}(\Gamma_0(N), \chi)$ is the space of cusp forms of weight $k$ and multiplier system $\chi$. 
Define the \emph{double Eichler integral}
$$
	I_{f_1, f_2}(\tau) \!:=\! \int_{\overline\tau}^{i \infty}  f_1(w_1)(w_1-\tau)^{k_1-2}\int_{w_1}^{i \infty} f_2(w_2)(w_2-\tau)^{k_2-2}dw_2 dw_1, \ \tau \in \mathbb H^-\!:=\!\{\tau\in\C:\im(\tau)<0\}.
$$
This integral was studied in Subsection 5.1 of \cite{BKM} (up to extra $i$-powers) as a function in the upper half-plane $\mathbb H:=\{\tau\in\mathbb C:\im(\tau)>0\}$, and as a result, the branch cuts due to the factors $(w-\tau)^{k-2}$, with $k$ half-integral, played a crucial role. In \cite{BKM}, the ``obstruction of modularity" of $I_{f_1, f_2}$ was computed and some modified quantum modularity was proved \cite[Theorem 5.1]{BKM}. The double iterated integral is the double integral counterpart of an integral associated with the weight $\frac32$ unary theta functions considered in \cite[Theorem 1.5]{BN}.

In this paper, we develop a cohomological framework in which both classes of objects can be incorporated. This framework offers an interpretation for the different behaviour of the obstruction to modularity of $I_{f_1, f_2}$ in the upper and lower half-plane and allows to derive some applications for values of $L_{f_1, f_2}(\frac ab; s_1, s_2)$. The cohomological construction is described in detail in \Cref{coh}, however, intuitively, a {\it cocycle of depth $n$} can be thought of as a map that fails to be a standard cocycle by a sum of products of cocycles of depth less than $n$. If $M$ is a $\mathbb C[G]$-module, then we denote the space of $n$-cocycles (resp. coboundaries, cohomology classes) of depth $k$ by $Z^n_{(k)}(G, M)$ (resp. $B^n_{(k)}(G, M), H^n_{(k)}(G, M)$). 
Two essential differences from previous cohomological interpretations of Manin's iterated Mellin transforms and multiple $L$-series are the following:
\begin{enumerate}[leftmargin=*]
	\item In previous works, the lower-depth terms were factored out, whereas they are visible by our approach and play an active role in our applications.
	\item Most of the earlier interpretations (e.g. \cite{BC,B2,CHO,MI}) were based on (standard) noncommutative cohomology, whereas our version is commutative.
\end{enumerate}

An approach closer to ours is that of \cite{C}, however the underlying congruence group in that work was of level $1$ and, more importantly, the functions here are required to have one variable. The first main application of our approach is as follows; the precise statement is given in \Cref{BasThgen}.

\begin{theorem}\label{cohIntr} Let $k_1, k_2\in\N+\frac12$ and $f_1, f_2$ cusp forms for $\Gamma_0(N)$, of weight $k_1, k_2$, respectively. Then, for $\gamma \in \Gamma_0(N)$, the function $r_{f_1, f_2}(\gamma)$ 
has an analytic continuation to a region containing $\mathbb H^- \cup \mathbb R \cup \{i \infty\}$ suitably truncated at neighborhoods of finitely many cusps.\footnote{For an example, see the region $\Omega$ in \Cref{fig-excnbh}.} The map sending $\gamma$ to the function $r_{f_1, f_2}(\gamma)$ satisfies the following \emph{$1$-cocycle of depth $2$}:
\begin{equation*}
	r_{f_1, f_2}(\gamma_{1}\gamma_{2})- r_{f_1, f_2}(\gamma_{1})|_{4-k_1-k_2}\gamma_{2}- r_{f_1, f_2}(\gamma_{2}) =r_{f_1}(\gamma_{1})|_{k_1}\gamma_{2} \, \cdot r_{f_2}(\gamma_{2}) \quad \text{for all $\g_1, \g_2 \in \G_0(N).$}
\end{equation*}
\end{theorem}

The second application pertains to integer weight forms and is analogous to relations of multiple zeta values in terms of ``lower order" multiple zeta values (see Section 3 of \cite{Za} for a brief overview). For depth $2$ it can be stated as follows; the proof is given in \Cref{depth2}.
\begin{theorem}\label{lincomb} For $\gamma_1, \gamma_2 \!\in\! \Gamma$, there exist $k_1\!+\!k_2\!-\!3$ non-trivial $\mathbb Q$-linear combinations of elements of $\{\Lambda_{f_1, f_2}\hspace{-0.1cm}\left((\gamma_1 \gamma_2)^{-1}i \infty, n, m\right)\hspace{-0.1cm}, \Lambda_{f_1, f_2}\hspace{-0.1cm}\left(\gamma_1^{-1}i \infty, n, m\right)\hspace{-0.1cm},  \Lambda_{f_1, f_2}\hspace{-0.1cm}\left(\gamma_2^{-1}i \infty, n, m\right)\hspace{-0.1cm}: n \in [1, k_1-1], m \in\ [1,  k_1+k_2\!-\!3-n] \}$, each of which is a $\mathbb Q$-linear combination of $\Lambda_{f_1}(\gamma_1^{-1}i \infty, n) \Lambda_{f_2}(\gamma_2^{-1}i \infty, m)$ ($1 \le n \le k_1-1$, $1 \le m \le k_2-1$).
\end{theorem}
The paper is organised as follows. In \Cref{cohom}, we discuss the preliminaries that are used in the sequel. In particular, we recall the basics of group cohomology, of the Eichler--Shimura--Manin theory and we summarise a construction of coefficient modules from \cite{BCD} on which our cohomological interpretation is based. In \Cref{II}, we introduce the ``master'' iterated integral that is then specialised. We also define a generalisation of multiple $L$-series twisted by additive characters. The behaviour of our iterated integrals in the special case of depth $2$ and integral weight is the focus of Section \ref{integral}, where applications to values of multiple $L$-series with additive twists are obtained. In \Cref{hiw}, we study the case of half-integral weight starting with a unified framework for the previous work on depth $1$ integrals \cite{BKM,BN}, and then extending it to depth $2$. In the final section, we address the case of general depth in both the integral and half-integral cases. We first determine the analogue of the cocycle condition for iterated integrals (\Cref{propfinn}) and formalise it into a generalisation of the standard cohomology (\Cref{mainDef}). In Proposition \ref{cdim1}, we show that, for a $\Gamma_0(N)$-module $M$, we have
$H_{(k)}^n(\Gamma_0(N), M)=0$, if $n \ge 2$.
\section*{Acknowledgements}
The first author has received funding from the European Research Council (ERC) under the European Union’s Horizon 2020 research and innovation programme (grant agreement No. 101001179).
\section{Preliminaries}\label{cohom}
We first recall the cohomological formalism, mainly from \cite{Br}, that is used below and some preliminaries mainly on notation. We next describe a space which plays the role of coefficient module in our cohomological interpretation leading up to Theorem \ref{cohIntr}. This space was constructed in \cite{BCD} and is based on work by Bruggeman--Lewis--Zagier \cite{BLZ1, BLZ2} on the cohomology of Maass wave forms.

\subsection{Preliminaries on group cohomology and actions of subgroups of the modular group.}
Let $G$ be a group and let $M, L$ be two right $\mathbb C [G]$-modules. We denote the action by ``$\circ$". For $n \in \mathbb N_0$, we consider {\it the space of $n$-cochains for $G$ with coefficients in $M$}
$$
	C^n(G, M):=\{\sigma: G^n \to M\}.
$$
By convention, we set
$C^n(G, M):=0$ for $n<0.$
We also use the formalism of ``bar resolution" for the differential $d^n\colon C^n(G, M) \to
C^{n+1}(G, M)$
\begin{multline}
d^n \sigma(g_1, \dots, g_{n+1}):=\sigma(g_2,\dots,g_{n+1}) \circ g_1
\\
+\sum_{j=1}^n (-1)^j \sigma(g_1,\dots, g_{j+1}g_j, \dots, g_{n+1})
+(-1)^{n+1}\sigma(g_1,\dots,g_{n}),\label{differential}
\end{multline}
where the notation $\sigma(g_1,\dots, g_{j+1}g_j, \dots, g_{n+1})$ means that the $j$-th component of $(g_1,\ldots,g_{n+1})$ is replaced by $g_{j+1}g_j$.
We define the groups $Z^n(G, M)$, $B^n(G, M)$, and $H^n(G, M)$ as usual
$$
	Z^n(G, M):=\operatorname{ker}\left(d^n\right), \quad B^n(G, M):=\operatorname{im}\left(d^{n-1}\right), \qquad H^n(G, M):=Z^n(G, M)/B^n(G, M).
$$
If $n=1$, then $r\in Z^1(G, M)$ is characterised by the $1$-cocycle condition
\begin{equation}\label{1cocy}
	r(\gamma_2 \gamma_1)=r(\gamma_2) \circ \gamma_1+r(\gamma_1), \qquad \text{for $\gamma_1, \gamma_2 \in G$}.
\end{equation}

We next recall the cup product construction (see, e.g., \cite[Subsection V.3]{Br}). First, we let $G$ act on $M \otimes L$ by diagonal action. To avoid overburdening the notation we use ``$\circ$'' for all actions.
We consider the {\it cup product map} $\cup\colon C^m(G, M) \otimes C^n(G, L) \to C^{m+n}(G, M \otimes L)$, defined by
\begin{equation*}
	\left ( \phi_1 \cup \phi_2 \right )(g_1, g_2, \dots, g_{m+n}):=(-1)^{mn}  \phi_1(g_{n+1}, \dots, g_{n+m}) \circ g_n \dots g_1\, \otimes \phi_2(g_1, \dots, g_n)
\end{equation*}
and the map
\begin{equation}\label{mu}\mu: C^*(G, M) \otimes C^*(G, L) \to C^*(G, M \otimes L)\end{equation}
induced by $\mu(\sigma_1 \otimes \sigma_2)=\sigma_1 \cup \sigma_2$.

We next recall a specific action which we use in the sequel. Let $k \in \frac12\mathbb Z$ and let $\Gamma$ be a subgroup of $\SL_2(\mathbb R)$.
Suppose that $\chi$ is a multiplier system of weight $k$ for $\Gamma$. Let $g$ be a function defined on a subset $C$ of $\mathbb C$ which is closed under the action of $\Gamma$. Then, for $\gamma=\big(\begin{smallmatrix}a&b\\c&d\end{smallmatrix}\big) \in \Gamma$ and $\tau \in C$, we have
\begin{equation*}
	g|_{k, \chi} \gamma(\tau):=\chi^{-1}(\gamma)j(\gamma, \tau)^{-k} g(\gamma \tau), \quad \text{where $j(\gamma, \tau):=c \tau+d$}.
\end{equation*}
This action is extended to $\mathbb C[\Gamma]$ by linearity.
We let $S_k(\Gamma, \chi)$ (resp. $M_k(\Gamma, \chi)$) denote the space of cusp (resp. modular) forms of weight $k$ and multiplier system $\chi$ for $\Gamma.$ If $\chi$ is the trivial character, then we simplify the notation to $|_k$,  $S_k(\Gamma)$, and $M_k(\Gamma)$, respectively.

An identity we use below is the following. As above, let $\chi$ be a multiplier system of weight $k \in \frac12 \mathbb Z$ for $\Gamma$, Then for $f\in S_k(\Gamma,\chi)$ we have
	\begin{equation}\label{modul}\frac{\chi^{-1}(\gamma) f(\gamma w)}{(\gamma w-\gamma\tau)^{2-k}j(\gamma, \tau)^{2-k}}d(\gamma w)=
	\frac{f(w)}{(w-\tau)^{2-k}}dw.
	\end{equation}

	Finally, we fix the notation $\tau=\tau_1+i\tau_2$ ($\tau_1, \tau_2 \in \mathbb R$) and we note that, throughout, the branch of the implicit logarithm is the principal branch, so that $-\pi< \text{Arg}(\tau) \le \pi$.

\subsection{Eichler--Shimura--Manin theory}
We recall some basics of the Eichler--Shimura--Manin theory which we need; for more details see, e.g. Subsection 11.8--11.11 of \cite{CS}. Let $k \in 2 \mathbb N$ and
$\Gamma\in\{\Gamma_0(N),\langle \Gamma_0(N),W_N \rangle\}$, where $W_N:=\begin{psmallmatrix}
	0 & -1\\
	N & 0
\end{psmallmatrix}$ is the {\it Fricke involution}. We consider the space $\mathbb C_{k-2}[\tau]$ of polynomials of degree at most $k-2$ viewed as a $\Gamma$-module under the action $|_{2-k}$. Furthermore let $\varphi$ be the map sending $(\overline f, g) \in \overline{S_k(\Gamma)} \oplus M_k(\Gamma)$ to the assignment
\begin{equation}\label{DefrPrel}
	\gamma \mapsto \int_{\gamma^{-1}i\infty}^{i \infty} \overline{f(w)}(\overline w-\tau)^{k-2}d \overline w+\int_{\gamma^{-1}i}^{i } g(w)(w-\tau)^{k-2}dw.
\end{equation}
The map $\varphi$ induces the \emph{Eichler--Shimura isomorphism} (see e.g., \cite[Theorem 11.8.5]{CS})
$$
	\overline{S_k(\Gamma)} \oplus M_k(\Gamma) \cong H^1(\Gamma, \mathbb C_{k-2}[\tau]).
$$
An application of this is Manin's Periods Theorem, a generalised version of which is the following:
\begin{proposition}
	\label{Razar} {\rm(\cite[Theorem 4]{R})} Let $f \in S_k(\Gamma_0(N))$ be a normalised Hecke eigenform and $K_f$ be the subfield of $\mathbb C$ generated by the Fourier coefficients of $f$. Then there are $\omega_1(f), \omega_2(f) \in \mathbb C$ with
	$$
		\Lambda_f\left(\gamma^{-1}\infty, n\right) \in K\omega_1(f)+K \omega_2(f) \qquad \text{for all $n\in\{1, \dots k-2\}$ and all $\gamma \in \Gamma_0(N)$.}
	$$
\end{proposition}

\subsection{The space of (smooth, excised) semi-analytic vectors}\label{DefMain}
Let $k \in \mathbb Z + \frac12$ and let $\G$ be a subgroup of $\SL_2(\mathbb R)$. However, we do not consider the cusps $\mathfrak{a}$ of $\G$ as equivalent classes, but simply as elements in $\Q\cup\{i\infty\}$. For an open subset $U \subset \mathbb{P}^1(\mathbb{C})$ not containing $i$,
we consider the map sending a function $f: U \to \C$
 to $\text{prj}_{2-k}(f): U \to \C$ defined by
\[  \text{prj}_{2-k} (f)(\tau):=(i-\tau)^{2-k}f( \tau).
\]
This map was used in \cite{BCD} to move between the ``projective'' and the ``plane'' model. In practice, it serves to simplify some computations.

For an open subset $U \subset \mathbb{P}^1(\mathbb{C})$,
we let $\hol(U)$ denote the space of holomorphic functions on $U$ and we define the space of {\it semi-analytic vectors} associated to $\{\mathfrak{a}_1, \dots, \mathfrak{a}_n\} \subset \mathbb{P}^1(\Q):=\Q \cup \{i\infty\}$ as
\begin{equation*}
D_{2-k}^{\omega}[\mathfrak{a}_1, \mathfrak{a}_2, \dots, \mathfrak{a}_n]:= \text{prj}_{2-k}^{-1}\indlim \hol(\Omega),
\end{equation*}
where $\Omega$ runs over the open sets of $\mathbb P^1_{\mathbb C}$ containing $\mathbb H^- \cup \mathbb P^1_{\mathbb R} \setminus \{\mathfrak{a}_1, \mathfrak{a}_2, \dots, \mathfrak{a}_n\}$. We call $\mathfrak{a}_j$ \emph{singular points}.
We next define the space of {\it smooth semi-analytic vectors} associated with $\{\mathfrak{a}_1, \dots, \mathfrak{a}_n\}$ as
$$D_{2-k}^{\omega, \infty}[\mathfrak{a}_1, \mathfrak{a}_2, \dots, \mathfrak{a}_n]:=
D_{2-k}^{\omega}[\mathfrak{a}_1, \mathfrak{a}_2, \dots, \mathfrak{a}_n]\cap
\text{prj}_{2-k}^{-1}\left(C^{\infty}\left(\mathbb H^- \cup \mathbb P^1_{\mathbb R}\right)\right).$$
A function $f$ belongs to $D^{\omega}_{2-k}[\mathfrak{a}_1, \dots, \mathfrak{a}_n]$ iff prj$_{2-k}(f)$ is holomorphic in some neighborhood of $\mathbb H^- \cup \mathbb P^1_{\mathbb R} \setminus \{\mathfrak{a}_1, \mathfrak{a}_2, \dots, \mathfrak{a}_n\}$.
On the other hand, $f$ belongs to $D^{\omega, \infty}_{2-k}[\mathfrak{a}_1, \dots, \mathfrak{a}_n]$ iff prj$_{2-k}(f)$ is holomorphic in some neighborhood $\mathbb H^- \cup \mathbb P^1_{\mathbb R} \setminus \{\mathfrak{a}_1, \mathfrak{a}_2, \dots, \mathfrak{a}_n\}$ and smooth on $\mathbb H^- \cup \mathbb P^1_{\mathbb R}$. This means that even though $f$ may not be holomorphic at $\mathfrak{a}_j$, it has an asymptotic series exists there: For any $N \ge 1$, we have, as $t \to \mathfrak{a}_j$
\begin{equation}\label{asym}
	f(t)=\sum_{n=0}^{N-1}a_n(t-\mathfrak{a}_j)^n+O\left((t-\mathfrak{a}_j)^N\right) \, \, \text{if $\mathfrak{a}_j \in \mathbb Q$,} \quad f(t)=\sum_{n=0}^{N-1}a_n t^{-n}+O\left(t^{-N}\right) \, \, \text{if $\mathfrak{a}_j=\infty$.}
\end{equation}

We finally define the spaces of excised semi-analytic vectors. Let $\mathfrak{a} \in \mathbb{P}^1(\Q):=\Q \cup \{i\infty\}$ be a cusp of $\Gamma$.
Suppose that $M(i\infty)=\mathfrak{a}$ for some $M \in$ $\SL_2(\Z).$ For $a,\varepsilon \in \mathbb{R}^+$, set
\[
  V_{\mathfrak{a}}(a, \varepsilon):= \{M\tau\in \mathbb H: |\tau_1| \le a, \tau_2 >\varepsilon\}.
\]
For a finite set $E$ of cusps of $\SL_2(\Z)$
a set $\Omega \subset \mathbb{P}^1(\mathbb{C})$ is called an {\it
$E$-excised neighbourhood} of $\mathbb H^-\cup \mathbb{P}^1(\mathbb{R})$ if there
exists a neighbourhood $U$ of $\mathbb H^-\cup \mathbb{P}^1(\mathbb{R})$ and
pairs of positive reals $(a_{\ca}, \varepsilon_{\ca})$ ($\ca \in E$) such
that $U \backslash \bigcup_{\mathfrak{a} \in E} V_{\mathfrak{a}}(a_{\ca},
\varepsilon_{\ca}) \subset \Omega$. An example is shown in Figure \ref{fig-excnbh} (note that part of $\Omega$ lies in $\mathbb H$).

\begin{figure}[ht]
\[\setlength\unitlength{.9cm}
\begin{picture}(10,5)(0,0)
\put(0,0){\line(1,0){10}}
\put(3.1,.1){$\mathfrak{a}_1$}
\put(5.1,.1){$\mathfrak{a}_2$}
\put(4,.1){$\Omega$}
\put(.5,-.4){$\Omega$}
\put(7,-.4){$\Omega$}
\put(9.1, 2){$\Omega$}
\put(3,-1.5){$\Omega$}
\put(.1, 2){$\Omega$}
\put(.1, .1){$\Omega$}
\put(9,.1){$\Omega$}
\put(3,0){\circle*{.1}}
\put(5,0){\circle*{.1}}
\put(3,.8){\vector(0,-1){.5}}
\put(2.8,1.2){$V_{\mathfrak{a}_1}(a_{\mathfrak{a}_1}, \varepsilon_{\mathfrak{a}_1})$}
\put(4,3){$V_{i \infty}(a_{i \infty}, \varepsilon_{i \infty})$}
\put(4.2,3.5){\vector(0,1){.6}}
\put(6,2){not in $\Omega$}
\thicklines
\put(2.5,0){\oval(1,1)[rt]}
\put(3.5,0){\oval(1,1)[lt]}
\put(4.5,0){\oval(1,1)[rt]}
\put(5.5,0){\oval(1,1)[lt]}
\put(3.5,.53){\line(1,0){1}}
\put(1,1){\oval(1,1)[lb]}
\put(1,.5){\line(1,0){1.5}}
\put(.5,1){\line(0,1){3}}
\put(8.5,1){\oval(1,1)[rb]}
\put(9,1){\line(0,1){3}}
\put(5.5,.5){\line(1,0){3}}
\end{picture}
\]
\vskip .5in
\caption{An $\{i\infty,\mathfrak{a}_1,\mathfrak{a}_2\}$-excised
neighbourhood.}\label{fig-excnbh}
\end{figure}
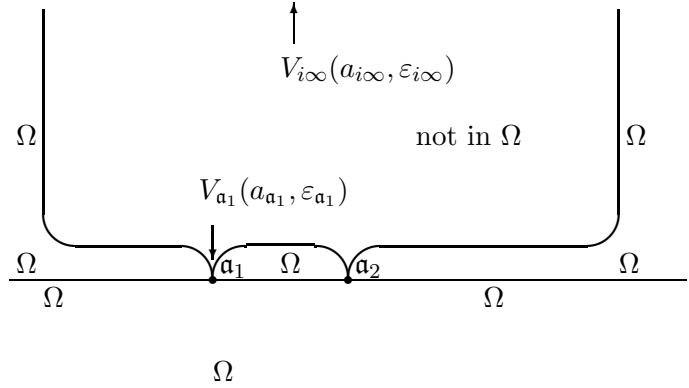
Define the space of {\it excised semi-analytic vectors}
associated to a finite set $\{\mathfrak{a}_1, \mathfrak{a}_2,
\dots, \mathfrak{a}_n\}$ of cusps
\begin{equation*}
\De_{2-k}[\mathfrak{a}_1, \mathfrak{a}_2, \dots, \mathfrak{a}_n]:= \text{prj}_{2-k}^{-1}\indlim\hol(\Omega),
\end{equation*}
where $\Omega$ runs over all $\{\mathfrak{a}_1, \mathfrak{a}_2, \dots,
\mathfrak{a}_n\}$-excised
neighbourhoods. A function $f$ belongs to $\De_{2-k}[\mathfrak{a}_1,$ $\dots,\mathfrak{a}_n]$ iff prj$_{2-k}(f)$ is holomorphic in some $\{\mathfrak{a}_1, \dots,\mathfrak{a}_n\}$-excised
neighbourhood. For example, if, for a function $f$, prj$_{2-k}(f)$ is holomorphic in the region $\Psi$ shown in Figure \ref{fig-sp}, then $f \in \De_{2-k}[-\frac dc]$, since $\Psi$ is a $\{-\frac dc\}$-excised neighborhood.
\vskip -12mm
		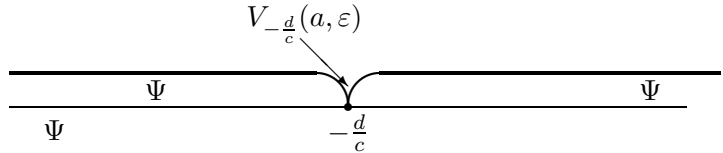
\begin{figure}[ht]
			\[
			\setlength\unitlength{.9cm}
			\begin{picture}(10,3)(0,0)
			\put(0,0){\line(1,0){10}}
			\put(4.7,-.5){$-\frac dc$}
			\put(2,.1){$\Psi$}
			\put(.5,-.5){$\Psi$}
			\put(9.3,.1){$\Psi$}
			\put(5,0){\circle*{.1}}
			\put(4.3, 1){\vector(1,-1){.7}}
			\put(3.5,1.2){$V_{-\frac dc}(a
				, \varepsilon
				)$}
			\thicklines
			\put(4.5,0){\oval(1,1)[rt]}
			\put(5.5,0){\oval(1,1)[lt]}
			\put(5.5,.5){\line(1,0){5}}
			\put(3.5,.5){\line(1,0){1}}
			\put(0,.5){\line(1,0){4}}
			\end{picture}
			\]
			\vskip 0cm \caption{A $\{-\frac dc\}$-excised neighbourhood, for $c \neq 0$.}\label{fig-sp}
		\end{figure}
We are now ready to define the following key spaces, given as the inductive limits
\begin{equation*}\label{defIL0}
 D^{\omega}_{2-k} :=\indlim D^{\omega}_{2-k}[\mathfrak{a}_1, \dots, \mathfrak{a}_n], \, \, \,  D^{\omega, \infty}_{2-k}:=\indlim D^{\omega, \infty}_{2-k}[\mathfrak{a}_1, \dots, \mathfrak{a}_n], \, \, \, \De_{2-k}:=\indlim \De_{2-k}[\mathfrak{a}_1, \dots, \mathfrak{a}_n],
\end{equation*}
where $\{\mathfrak{a}_1, \mathfrak{a}_2, \dots, \mathfrak{a}_n\}$ ranges
over all finite sets of cusps of $\SL_2(\Z)$. An important fact is that, by Proposition 1.14 of \cite{BCD}, these spaces is closed under the action $|_{2-k, \chi}$. The space that we mainly use as our coefficient module is then defined as
$$D^{\omega, \infty, \text{exc}}_{2-k}:=D^{\omega, \infty}_{2-k} \cap D^{\omega,\text{exc}}_{2-k}.$$
\begin{examples}\hspace{0cm}{\rm
	\begin{enumerate}[label=\rm(\arabic*),leftmargin=*]
		\item If prj$_{2-k}(f)$ is holomorphic in $\mathbb C\setminus\{\tau:|\tau_2|>M\}$ for some $M>0$, then $f \in \De_{2-k}$ because $f \in \De_{2-k}[\infty]$.
 If moreover $f$ has an asymptotic expansion as in \eqref{asym}, for $t \to \infty$, then $f \in D^{\omega, \infty, \text{exc}}_{2-k}.$
		\item Suppose that $\mathfrak a_1=1, \mathfrak a_2=2$ as in Figure \ref{fig-excnbh}. If prj$_{2-k}(f)$ is holomorphic in $\Omega$ as in Figure \ref{fig-excnbh}, then $f \in \De_{2-k}[1, 2, \infty]$ and thus $f \in \De_{2-k}$.  If, in addition, $f$ has asymptotic expansions as in \eqref{asym} for $t \to \mathfrak{a}_1,$ $t \to \mathfrak{a}_2$ and $t \to \infty,$ then $f \in D^{\omega, \infty, \text{exc}}_{2-k}.$
	\end{enumerate}}
\end{examples}

\section{Iterated integrals and multiple $L$-series}\label{II}
In this section, we first introduce a map \eqref{master} from which the main objects we consider are derived. We then define the multiple $L$-series with additive twists and show how they characterise the values of the map \eqref{master} at certain integers.

Let $\Gamma\in\{\Gamma_0(N),\langle \Gamma_0(N),W_N \rangle\}$. For $k_j \in 2\mathbb N$, $f_j \in S_{k_j}(\Gamma)$, and $s_j \in \mathbb C$, where $j\in\{1, \dots r\}$, Manin, in \cite[Definition 2.2]{MI}, introduced the following \emph{iterated Mellin transform}
$$M(f_1,\dots f_r; s_1, \dots s_r)=(-1)^r\int_{0}^{i \infty} \int_{w_1}^{i \infty} \dots \int_{w_{r-1}}^{i \infty} \prod_{j=1}^rf_j(w_j)w_j^{s_j} dw_r \dots dw_1. $$
Generalising the classical case, for $s_j\in\N$, he expressed it as a linear combination of values of a {\it multiple $L$-series} defined as
$$L_{f_1, \dots, f_r}(s_1,\dots, s_r)=\sum_{n_1, \dots, n_r \ge 1}\frac{c_1(n_1)\dots c_r(n_r)}{(n_1+\dots+n_r)^{s_1}\dots (n_{r-1}+n_r)^{s_{r-1}}n_r^{s_r}},$$ where $c_j(n)$ denotes the $n$-th Fourier coefficient of $f_j.$
These have been studied for example in \cite{C,Pr}, where the emphasis was on their values at positive integer $s_j$, in analogy with the multiple zeta values. We do not use the term ``multiple modular values'' for those values because this term has been employed by Brown for a vast generalisation of such values which have been studied motivically in \cite{B1,B2}.

In each of the papers \cite{C,Pr}, the study of the multiple $L$-series was based on an expression in terms of a variation of Manin's iterated Mellin transform. Each of these variations was designed to focus on a different aspect of the multiple $L$-series. For instance, in \cite{C} it was used to derive relations among values of multiple $L$-series, mirroring the extensively studied subject of relations among multiple zeta values. For this, a first, together with \cite{BC}, cohomological study of multiple $L$-series and their values is undertaken and analogues of Eichler integrals are defined.

Our principal object in this section unifies and extends the constructions of those works. We consider $k_j \in \frac12 \mathbb Z$, $f_j \in S_{k_j}(\Gamma, \chi)$, and $s_j \in \mathbb C$, where $j\in\{1, \dots r\}$. Then, we let $r^*_{f_1, \dots f_r}(s_1, \dots, s_r) (\gamma)(\tau)$ be the map
sending $\gamma \in \Gamma_0(N)$ to the function given, for $\tau \in \mathbb H^- \cup \mathbb Q \cup \{i \infty \}$
\begin{equation}\label{master}
r^*_{f_1, \dots f_r}(s_1, \dots, s_r)(\gamma)(\tau)=\int_{\gamma^{-1}i\infty}^{i \infty} \int_{w_1}^{i \infty} \dots \int_{w_{r-1}}^{i \infty} \prod_{j=1}^rf_j(w_j)(w_j-\tau)^{s_j} dw_r \dots dw_1.
\end{equation}
We can retrieve Manin's iterated Mellin transform. Namely, for $\Gamma=\langle \Gamma_0(N), W_N \rangle$, we have
$$r^*_{f_1, \dots f_r}(s_1, \dots, s_r)(W_N)(0)=(-1)^r M(f_1,\dots f_r; s_1, \dots s_r).$$
For $\tau \in \mathbb H^-$, the basic object $R_{r}$ of \cite{BC} is
$$R_r(f_1,\dots f_r; s_1, \dots s_r; \gamma^{-1}i \infty, \infty; \tau):=(-1)^r r^*_{f_1, \dots f_r}(s_1, \dots, s_r)(\gamma)(\tau).$$
If $k_j \in 2 \mathbb N$, and $N=1$, then the polynomial $R_{2, r}(\gamma; (\tau_1, \dots, \tau_r))$ of \cite{C}, for $\tau_1=\dots=\tau_r=\tau$ equals
$$ R_{2, r}(\gamma; (\tau, \dots, \tau))=r^*_{f_1, \dots f_r}(k_1-2, \dots, k_{r}-2)(\gamma)(\tau).$$

Our counterpart of the integral expressions for the multiple $L$-series in the above works pertains to the additive twist of the multiple $L$-series. For $\operatorname{Re}(s_j) \gg 1$, and $\frac ab \in \mathbb Q$, we set
$$L_{f_1, \dots, f_r}\lp \frac ab; s_1,\dots, s_r \rp:=\sum_{n_1, \dots, n_r \ge 1}\frac{c_1(n_1)\dots c_r(n_r)e^{2 \pi i (n_1+\dots +n_r)\frac ab}}{(n_1+\dots+n_r)^{s_1}\dots (n_{r-1}+n_r)^{s_{r-1}}n_r^{s_r}}$$ 
and we define the {\it completed multiple $L$-series with additive twists}
$$\Lambda_{f_1, \dots, f_r}\lp \frac ab; s_1,\dots, s_r \rp:=
\int_{\frac ab}^{i \infty} f_1(w_1) \lp w_1-\frac ab \rp ^{s_1-1} \int_{w_1}^{i \infty} \dots \int_{w_{r-1}}^{i \infty} \prod_{j=2}^rf_j(w_j)(w_j-w_{j-1})^{s_j-1}dw_j \, dw_1.$$
We set
$$A_{n_1, \dots n_r}^{[m_1,\dots m_r]}:=\binom{m_r}{n_r}\binom{m_r+m_{r-1}-n_r}{n_{r-1}} \dots \binom{m_r+\dots +m_2-n_r-\dots-n_2}{n_1}.$$
 We then have the following result which justifies the name ``completed $L$-series''.
\begin{proposition}\label{MellLser} \hspace{0cm}
\begin{enumerate}[leftmargin=*,label=\rm(\arabic*)]
	\item For $\operatorname{Re}(s_j) \gg 1$, we have
	$$
		\Lambda_{f_1, \dots, f_r}\lp \frac ab; s_1,\dots, s_r \rp=	\frac{\Gamma(s_1)\dots \Gamma(s_r)}{(-2\pi i)^{s_1+\dots s_r}}L_{f_1, \dots, f_r}\lp \frac ab; s_1,\dots, s_r \rp.
	$$
	\item Fix $m_j \in 2 \mathbb N$. For $\gamma \in \G$, we have
	\begin{equation*}
		 r^*_{f_1, \dots, f_r}(m_1, \dots, m_r)(\gamma)(\tau)=\sum_{n=0}^{m_1+\dots+m_r}	\mathcal L_{f_1, \dots f_r} \lp \gamma^{-1}(i\infty) ; n \rp \lp \gamma^{-1}(i\infty)-\tau \rp^n,
	\end{equation*}
where
	$$
		\mathcal L_{f_1, \dots f_r} \lp \frac ab; n \rp:= \sum_{\substack{n_2, \dots, n_r \ge 0 \\ n_1=m_1+\dots+m_r-n_2-\dots-n_r-n}} A_{n_1, \dots n_r}^{[m_1,\dots m_r]} \Lambda_{f_1, \dots f_r} \lp \frac ab; n_1+1, \dots, n_r+1 \rp.
	$$
\end{enumerate}
\end{proposition}
\begin{proof}
(1) We choose the lines of integration $w_1=\frac ab+it_1$ and $w_j=w_{j-1}+it_j$ ($t_j \ge 0$) and the conclusion follows as in the one-dimensional case (e.g. (2.10) of \cite{DHKL}).

\noindent(2) The Binomial Theorem applied to $(w_r-\tau)^{m_r}=((w_r-w_{r-1})+(w_{r-1}-\tau))^{m_r}$ implies that
\begin{multline*}r^*_{f_1, \dots, f_r}(m_1, \dots, m_r)(\gamma)(\tau)=\sum_{n_r \ge 0} \binom{m_r}{n_r}
\int_{\frac ab}^{i \infty} \dots \int_{w_{r-3}}^{i \infty}\prod_{j=1}^{r-2}f_j(w_j)(w_j-\tau)^{m_j}\\
\times \int_{w_{r-2}}^{i \infty}\!\! f_{r-1}(w_{r-1})(w_{r-1}-\tau)^{m_{r-1}+m_r-n_{r}}\!\int_{w_{r-1}}^{i \infty}\!\! f_{r}(w_{r})(w_{r}-w_{r-1})^{n_{r}}dw_r \dots dw_1.
\end{multline*}
By induction, we deduce the claim.
\end{proof}

\section{The case of integral weight.}\label{integral}
In this section we apply the set up of Section 3 to the case of integral weight to derive results about values of multiple $L$-series with additive twists.

\subsection{Depth $1$.}\label{de1} The case of depth $1$ is the setting of the classical Eichler--Shimura--Manin theory outlined in \Cref{cohom}. Specifically, let $k \in 2 \mathbb N$ and $\mathbb C_{k-2}[\tau]$ the space of polynomials of degree at most $k-2$. The group $\Gamma\in\{\Gamma_0(N),\langle \Gamma_0(N), W_N \rangle\}$ acts on $\mathbb C_{k-2}[\tau]$ via $|_{2-k}$. We have the map sending $f \in S_k(\Gamma)$ to the map $r_{f}: \Gamma \to \mathbb C_{k-2}[\tau]$ given by
\begin{equation}\label{Defr}r_f(\gamma)(\tau):=r^*_f(k-2)(\gamma)(\tau)=\int_{\gamma^{-1}i\infty}^{i \infty} f(w)(w-\tau)^{k-2}dw.
\end{equation}
The map $r_f$ is a restriction of the map \eqref{DefrPrel} which induces the Eichler--Shimura isomorphism.

\subsection{Depth $2$.}\label{depth2} We next consider the case of two cusp forms $f_j \in S_{k_j}(\Gamma)$, $j\in\{1,2\}$, and the map sending $\gamma \in \Gamma$ to the following polynomial in $\mathbb C_{k_1+k_2-4}[\tau]$
$$
	r_{f_1, f_2}(\gamma)(\tau)\!:=\!r^*_{f_1,  f_2}(k_1-2, k_2-2)(\gamma)(\tau)\!=\!\int_{\gamma^{-1}i\infty}^{i \infty} \int_{w_1}^{i \infty} f_1(w_1)f_2(w_2)(w_1-\tau)^{k_1-2}(w_2-\tau)^{k_2-2}dw_2 dw_1.
$$
The group $\Gamma$ acts on $\mathbb C_{k_1+k_2-4}[\tau]$. The map $r_{f_1, f_2}$ is no longer a $1$-cocycle but it satisfies the following generalisation of the $1$-cocycle relation.
\begin{theorem}\label{2dimZ}
	Let $f_j \in S_{k_1}(\Gamma)$, $\gamma_j\in\Gamma$ ($j\in\{1,2\}$). We have
\begin{equation}\label{gen1coc}
	r_{f_1, f_2}(\gamma_{1}\gamma_{2})- r_{f_1, f_2}(\gamma_{1})|_{4-k_1-k_2}\gamma_{2}- r_{f_1, f_2}(\gamma_{2}) =r_{f_1}(\gamma_{1})|_{k_1}\gamma_{2} \, \cdot r_{f_2}(\gamma_{2}).
\end{equation}
\end{theorem}
\begin{proof}
The change of variables $w_2 \mapsto \gamma_2^{-1}w_2$ combined with \eqref{modul}, applied to $\chi_j = 1$ and $\gamma=\gamma_2$, implies that $r_{f_1, f_2}(\gamma_{1})|_{4-k_1-k_2}\gamma_{2}(\tau)$ equals
	\begin{equation*}
\int_{\gamma_{1}^{-1}i \infty}^{i \infty} f_1(w_1)(w_1-\gamma_2 \tau)^{k_1-2}j(\gamma_2, \tau)^{k_1-2} \int_{\gamma_2^{-1} w_1}^{\gamma_2^{-1}i \infty} f_2(w_2)(w_2-\tau)^{k_2-2}dw_2 dw_1.
	\end{equation*}
	Making the change of variables $w_1 \mapsto \gamma_2w_1$, we deduce that
	\begin{align*}
		&r_{f_1, f_2}(\gamma_1)|_{4-k_1-k_2}\gamma_2(\tau) = \int_{(\gamma_1 \gamma_{2})^{-1}i \infty}^{\gamma_2^{-1}i \infty} \int_{w_1}^{\gamma_{2}^{-1}i \infty} f_1(w_1)f_2(w_2)(w_1-\tau)^{k_1-2}(w_2-\tau)^{k_2-2}dw_2dw_1\\
		&\hspace{-0.1cm}=\lp\! \int_{(\gamma_1 \gamma_{2})^{-1}i \infty}^{i \infty}\int_{w_1}^{i \infty} \!\!\!+\! \int_{(\gamma_1 \gamma_{2})^{-1}i \infty}^{i \infty} \int_{i \infty}^{\gamma_{2}^{-1}i \infty} \!\!\!+\! \int_{i \infty}^{\gamma_2^{-1}i \infty} \hspace{-0.25cm} \int_{w_1}^{i \infty} \!\!\!+\! \int_{i \infty}^{\gamma_2^{-1}i \infty} \hspace{-0.25cm} \int_{i \infty}^{\gamma_{2}^{-1}i \infty} \!\rp\! \prod_{j=1}^2\!f_j(w_j)(w_j\!-\!\tau)^{k_j-2} dw_2 dw_1\\
		&\hspace{-0.1cm}=r_{f_1, f_2}(\gamma_1 \gamma_2)(\tau)-r_{f_1}(\gamma_1 \gamma_2)(\tau)r_{f_2}(\gamma_2)(\tau)-r_{f_1, f_2}(\gamma_2)(\tau) + r_{f_1}(\gamma_2)(\tau)r_{f_2}(\gamma_2)(\tau).
\end{align*}
	The conclusion follows from the $1$-cocycle relation \eqref{1cocy} for $r_{f_1}$.
\end{proof}
The relation in \Cref{2dimZ} can be interpreted as a ``$1$-cocycle relation up to lower depth terms" and this was the viewpoint taken in \cite{BC}. In contrast, in our perspective, we take those ``lower order terms" more explicitly into account. With \Cref{2dimZ}, we can now prove Theorem \ref{lincomb}.

\noindent{\it Proof of Theorem \ref{lincomb}.}
With Proposition \ref{MellLser} (2), the left-hand side of \eqref{gen1coc} becomes
\begin{multline}\label{LHS}\sum_{n=0}^{k_1+k_2-4} \Big ( \mathcal L_{f_1, f_2} \lp (\gamma_1 \gamma_2)^{-1}i \infty; n \rp \lp (\gamma_1 \gamma_2)^{-1}i \infty-\tau \rp^{n}\\
 -\mathcal L_{f_1, f_2} \lp \gamma_1^{-1}i \infty; n \rp \lp \gamma_1^{-1}i \infty-\gamma_2 \tau \rp^{n}j(\gamma_2, \tau)^{k_1+k_2-4}-
\mathcal L_{f_1, f_2} \lp \gamma_2^{-1}i \infty; n \rp \lp \gamma_2^{-1}i \infty-\tau \rp^{n} \Big ).
\end{multline}
The coefficients of this, as polynomials in $\tau$, are linear combinations of the form described in the statement of \Cref{lincomb}. The right-hand side of \eqref{gen1coc} is clearly a polynomial in $\mathbb C_{k_1+k_2-4}[\tau]$ whose coefficients are $\mathbb Q$-linear combinations of
$\Lambda_{f_1}(\gamma_1^{-1}i \infty, n) \Lambda_{f_2}(\gamma_2^{-1}i \infty, m)$ with $1 \le n \le k_1-1$ and $1 \le m \le k_2-1$. By comparing polynomial coefficients between the left- and the right-hand side, we deduce the claim. \qed

The proof of Theorem \ref{lincomb} gives the following corollary.
\begin{corollary}\label{L2dim} For  $\gamma_1, \gamma_2 \in \Gamma_0(N)\setminus\{I_2\}$, we have that
\begin{multline}\label{LHSc}\sum_{n=0}^{k_1+k_2-4} \Big ( \mathcal L_{f_1, f_2} \lp (\gamma_1 \gamma_2)^{-1}i \infty; n \rp \lp (\gamma_1 \gamma_2)^{-1}i \infty \rp^{n}\\
-\mathcal L_{f_1, f_2} \lp \gamma_1^{-1}i \infty; n \rp \lp \gamma_1^{-1}i \infty-\gamma_2 0 \rp^{n}j(\gamma_2, 0)^{k_1+k_2-4}-
 \mathcal L_{f_1, f_2} \lp \gamma_2^{-1}i \infty; n \rp \lp \gamma_2^{-1}i \infty \rp^{n} \Big )
\end{multline}
is a $\mathbb Q$-linear combination of $\Lambda_{f_1}(\gamma_1^{-1}i \infty, n) \Lambda_{f_2}(\gamma_2^{-1}i \infty, m)$ ($1 \le n \le k_1-1$, $1 \le m \le k_2-1$).
In particular, for each $k \in \mathbb N,$
\begin{equation}\label{LHScspec}
	\sum_{n=0}^{k_1+k_2-4} \lp -\frac{1}{N} \rp^n \lp \frac{1}{(1+k)^n}\mathcal L_{f_1, f_2} \lp \frac{-1}{N(k+1)}; n \rp -\mathcal L_{f_1, f_2} \lp \frac{-1}{N}; n \rp -\frac{1}{k^n}\mathcal L_{f_1, f_2} \lp \frac{-1}{Nk}; n \rp \rp
\end{equation}
is a $\mathbb Q$-linear combination of $\Lambda_{f_1}(\frac{-1}{N},n) \Lambda_{f_2}(\frac{-1}{Nk},m)$ ($1 \le n \le k_1-1$, $1 \le m \le k_2-1$). If $f_1, f_2$ are normalised Hecke eigenforms, then, for $\gamma_1, \gamma_2 \in \Gamma_0(N)\setminus\{I_2\}$ and $k \in \mathbb N$, \eqref{LHSc} and \eqref{LHScspec} belong to the $4$-dimensional $K_{f_1}K_{f_2}$-space generated by
$\omega_1(f_1) \omega_1(f_2)$, $\omega_1(f_1) \omega_2(f_2)$, $\omega_2(f_1) \omega_1(f_2)$, and $\omega_2(f_1) \omega_2(f_2)$.
\end{corollary}
\begin{proof}
	By comparing the constant term of \eqref{LHS} with that of the right-hand side of \eqref{gen1coc} we obtain the first claim. The second claim is deduced from the first one in the special case $\gamma_1=\begin{psmallmatrix}
	1 & 0\\
	N & 1
\end{psmallmatrix}$ and $\gamma_2=\gamma_1^k.$
The last claim is a consequence of Proposition \ref{Razar}.
\end{proof}
The comparison of leading terms of the polynomials gives the following.
\begin{corollary} For $\gamma_1, \gamma_2 \in \Gamma_0(N)\setminus\{I_2\}$, we have that
\begin{multline}\label{LHSl} \mathcal L_{f_1, f_2} \lp (\gamma_1 \gamma_2)^{-1}i \infty; k_1+k_2-4 \rp
-\sum_{n=0}^{k_1+k_2-4} (-1)^n \mathcal L_{f_1, f_2} \lp \gamma_1^{-1}i \infty; n \rp j(\gamma_2^{-1}, \gamma_1^{-1} \infty)^n c_{\gamma_2}^{k_1+k_2-4-n}
\\-
 \mathcal L_{f_1, f_2} \lp \gamma_2^{-1}i \infty; k_1+k_2-4 \rp
 = \lp \Lambda_{f_1}((\gamma_1 \gamma_2)^{-1}i \infty, 1)-
\Lambda_{f_1}(\gamma_2^{-1}i \infty, 1) \rp \Lambda_{f_2}(\gamma_2^{-1}i \infty, 1).
\end{multline}
Here $c_{\gamma_2}$ denotes the bottom left entry of $\gamma_2$.
In particular, for $k \in \mathbb N$, we have that
\begin{multline}\label{LHSlspec}\mathcal L_{f_1, f_2} \lp -\frac{1}{N(k+1)}; k_1+k_2-4 \rp -\sum_{n=0}^{k_1+k_2-4} (-k-1)^n (kN)^{k_1+k_2-4-n}\mathcal L_{f_1, f_2} \lp -\frac{1}{N}; n \rp\\ -\mathcal L_{f_1, f_2} \lp -\frac{1}{Nk}; k_1+k_2-4 \rp
=\lp\Lambda_{f_1} \lp -\frac1{N(k+1)}, 1 \rp-
\Lambda_{f_1} \lp -\frac1{Nk}, 1 \rp \rp \Lambda_{f_2} \lp -\frac1{Nk}, 1 \rp.
\end{multline}

If $f_1, f_2$ are normalised Hecke eigenforms, then, for $\gamma_1, \gamma_2 \in \Gamma_0(N)\setminus\{I_2\}$ and $k \in \mathbb N$, \eqref{LHSl} and \eqref{LHSlspec} belong to the $4$-dimensional $K_{f_1}K_{f_2}$-space generated by
$\omega_1(f_1) \omega_1(f_2), \dots, \omega_2(f_1) \omega_2(f_2)$.
\end{corollary}
\begin{proof}
	For the first claim we compare the leading terms of the two sides of \eqref{gen1coc}. To simplify the term $(\gamma_1^{-1}i \infty-\gamma_2 \tau)^{n}j(\gamma_2, \tau)^{k_1+k_2-4}=(\gamma_1^{-1}i \infty-\gamma_2 \tau)^{n}j(\gamma_2, \tau)^{n}j(\gamma_2, \tau)^{k_1+k_2-4-n}$ we employ \eqref{modul}. For the right-hand side we use the $1$-cocycle condition of $r_{f_1}$. The remaining two claims are deduced as in Corollary \ref{L2dim}.
\end{proof}
In the special case of $N=1$ and $\gamma_1=\gamma_2=W_1$, the conclusion of Theorem \ref{lincomb} was established in an explicit form in \cite[Theorem 3.2]{C}.
\section{The case of half-integral weight.}\label{hiw}
An essential difference between the cases of integral and half-integral weight is the domain of $r_{f_1, \dots, f_r}(\gamma)$. Whereas in the former case, $r_{f_1, \dots, f_r}(\gamma)$ can simply be treated as a polynomial in $\tau$, if the weights are non-integral, the range of $\tau$ must be chosen more carefully to ensure that the defining integral converges and is independent of the path of integration. This is automatic if $\tau$ belongs to $\mathbb H^-$ (and the cusps), however, for questions related to quantum modular forms, or to ``noncommutative boundary" (\cite[Subsection 0.2]{MI}), it is important to know if the function can be defined beyond $\mathbb H^-.$
In this section, we apply the general framework of \Cref{DefMain} to address this problem in a way that allows for a cohomological interpretation. We start by recalling some of the main objects of \cite{BKM,BN}, and show how they can be interpreted in our framework.

\subsection{Depth $1$}
Let $\Gamma=\Gamma_0(N)$ $4|N$, $k \in \mathbb Z+\frac12$, and $\chi$ a multiplier system of weight $k$. For $f \in S_{k}(\Gamma,\chi)$, the function $r_f(\gamma)$ defined by \eqref{Defr}, was studied (even for more general real $k$) for $\tau$ in the lower half-plane by Gunning, Kra, Knopp, and others (see \cite{Kn} and the references therein). From this perspective, it is possible to obtain isomorphisms in the spirit of Eichler--Shimura where the coefficient module of the classical Eichler theory is replaced by the space of holomorphic functions on $\mathbb H^-$ of ``polynomial growth" (the space $D_{2-k}^{-\infty}$ in the proof of Proposition \ref{BasTh} (2)). This approach corresponds more closely to that of \cite{BKM}, where the following lemma was proved (\cite[Lemma 2.3]{BKM}). To state it, for $f \in S_{k}(\Gamma,\chi)$, set
$$I_{f}(\tau) := \int_{\overline\tau}^{i \infty} \frac{f(w)}{(w-\tau)^{2-k}}dw.$$
\begin{lemma}\label{BKM23}
The function $I_{f}$ is defined on $\mathbb H^- \cup \mathbb Q$ and, for $\gamma \in \Gamma$ and $\tau \in \mathbb H^-$, we have
\begin{equation*}
	I_f|_{2-k, \chi}(I_2-\gamma)=r_{f}(\gamma).
\end{equation*}
\end{lemma}
An analogous proposition to Lemma \ref{BKM23} was proved in \cite[Theorem 1.5]{BN} in the case that $f$ is a weight $\frac32$ unary theta function. We state it here, but we slightly renormalise the function denoted $F_{j, N}$ used there, in order to stress the similarity with Lemma \ref{BKM23}. Namely for $1 \le j \le N-1$, let
$$f_{j, N}(\tau):=\frac{1}{2N}\sum_{\substack{n \in \mathbb Z \\ n \equiv j \hspace{-0.2cm} \pmod{2N}}}n q^{\frac{n^2}{4N}}$$
be a weight $\frac32$ unary theta function with some multiplier system $\{\psi_{j, r} \}$ and set
$$F_{j, N}(\tau):=\frac{i}{\sqrt{2N}}\sum_{\substack{n \in \mathbb Z \\ n \equiv j \hspace{-0.2cm} \pmod{2N}}} \sgn (n) q^{\frac{n^2}{4N}}.$$
\begin{proposition}\label{BN15}
For $\gamma =
\begin{psmallmatrix}
	a & b\\
	c & d
\end{psmallmatrix}
\in \SL_2(\mathbb Z)$ and $\tau \in \mathbb H$, there exists a path of integration of $r_{f_{j, N}}(\gamma)(\tau)$ avoiding the branch cut defined by $\sqrt{-i(z-\tau)}$ such that
\begin{equation*}
F_{j, N}(\tau)-\sgn(c\tau_1+d)\sum_{j=1}^{N-1}F_{r, N}|_{\frac12, \psi_{j, r}}\gamma(\tau)=r_{f_{j, N}}(\gamma)(\tau).
\end{equation*}
\end{proposition}
In contrast to Lemma \ref{BKM23}, the function \eqref{Defr} in Proposition \ref{BN15} is studied in the upper half-plane where considerations of branch cuts become crucial. The constructions of the last subsection allow a systematic way to account for the branch cuts in a way compatible with the action of the group so that we can interpret $r_f$ cohomologically.
In particular, the proposition can be reformulated to avoid reference to the dependence on the path of integration in its statement.
\begin{proposition}\label{BN15'} For $\gamma =
\begin{psmallmatrix}
	a & b\\
	c & d
\end{psmallmatrix}
\in \SL_2(\mathbb Z)$ there exists a $-\frac dc$-excised neighborhood $\Psi$ such that, for $\tau \in  \mathbb H \cap \Psi$, we have
\begin{equation*}
F_{j, N}(\tau)-\sgn(c\tau_1+d)\sum_{j=1}^{N-1}F_{r, N}|_{\frac12, \psi_{j, r}}\gamma(\tau)=r_{f_{j, N}}(\gamma)(\tau).
\end{equation*}
\end{proposition}

The discussion of the different approaches behind Lemma \ref{BKM23} and Proposition \ref{BN15} motivates the choice of the coefficient module in the next proposition. In a sense, Proposition \ref{BasTh} represents a unification of the above two approaches to studying $r_f(\gamma)$; part (1) essentially says that the function $r_{f}(\gamma)$, originally defined on $\mathbb H^-$, can be holomorphically continued ``slightly above" the real axis taking into account the existence of branch cuts. In addition it is done so that we remain within a $\Gamma$-invariant space and the resulting map is a $1$-cocycle. In part (2), the non-triviality of the cohomology class is proved by restricting to $\mathbb H^-$ (from the slightly larger domain where $r_f(\gamma)$ is defined in part (1)) and exploiting Knopp's version of the Eichler--Shimura isomorphism \cite[Theorem 1]{Kn}. This interplay between the two approaches is important throughout this paper.

\begin{proposition}\label{BasTh}\hspace{0cm}
\begin{enumerate}[leftmargin=*,label=\rm(\arabic*)]
	\item The map $\gamma \mapsto r_f(\gamma)$ induces a $1$-cocycle with coefficients in $D^{\omega, \infty, \textup{exc}}_{2-k}$. This induces a map
	$$
		\br_k: S_k(\Gamma, \chi) \longrightarrow H^1\left(\Gamma, D^{\omega, \infty, \textup{exc}}_{2-k}\right).
	$$
	\item  If $f \neq 0$, then the class $\br_k(f)=[r_f]$ in $H^1(\Gamma, D^{\omega, \infty, \textup{exc}}_{2-k})$ is non-trivial.
\end{enumerate}
\end{proposition}
\begin{proof}
\begin{enumerate}[label=\rm(\arabic*),wide, labelwidth=!, labelindent=0pt]
\item  This is the content of Proposition 2.6 (i) of \cite{BCD}.
\item
 Let $D_{2-k}^{-\infty}$ be the space of holomorphic functions $f$ on $\mathbb H^{-}$ such that, for $\tau \in \mathbb H^-$, we have $f(\tau)=O(\tau_2^{-C})+O(|\tau|^C)$, for some $C>0$. Note that this is a $\Gamma$-module. Furthermore, in (d) of the discussion of Definition 1.4 of \cite{BCD}, it was shown that $\text{prj}_{2-k}^{-1}(C^{\infty}(\mathbb H^- \cup \mathbb P^1_{\mathbb R})) \subset D_{2-k}^{-\infty} $ and hence
\begin{equation}\label{incl}D^{\omega, \infty, \text{exc}}_{2-k}=D^{\omega, \infty}_{2-k} \cap D^{\omega,\text{exc}}_{2-k} \subset D^{\omega, \infty}_{2-k} \subset \text{prj}_{2-k}^{-1}\left(C^{\infty}\left(\mathbb H^- \cup \mathbb P^1_{\mathbb R}\right)\right) \subset D_{2-k}^{-\infty}.\end{equation}
On the other hand, by theorems of Knopp and Knopp--Mawi \cite[Theorem 1]{Kn}, \cite[Theorem 2.1]{KM} (in the formulation of Theorem 2.7 of \cite{BCD}), \eqref{Defr} induces an isomorphism between $S_k(\Gamma, \chi)$ and $ H^1(\Gamma, D_{2-k}^{-\infty})$. We thus have the commutative diagram
\begin{equation}\label{eqcomm}
  \begin{tikzcd}
    S_k(\Gamma, \chi) \arrow{r}{\br_k} \arrow[swap]{dr}{\cong} & H^1\left(\Gamma, D_{2-k}^{\omega, \infty, \text{exc}}\right) \arrow{d}{\iota} \\
     & H^1\left(\Gamma, D_{2-k}^{-\infty}\right)
  \end{tikzcd}
\end{equation}
 where $\iota$ is the natural map induced by the inclusion \eqref{incl}. Since $\iota \circ \br_k$ is injective, $\br_k$ is injective too and, hence, $\br_k(f)=[r_f]$ is only trivial if $f \equiv 0$.
\end{enumerate}
\vspace*{-0.5cm}
\end{proof}
\begin{remarks}\hspace{0cm}
	\begin{enumerate}[leftmargin=*,label=\rm(\arabic*)]
		\item The range of $\br_k$ in Proposition 2.6  of \cite{BCD} is a smaller space (the parabolic cohomology group), however there is no loss of information by considering the full cohomology group.
		\item The assertion of Proposition \ref{BasTh} {\rm(1)} says that
		$$\operatorname{prj}_{2-k}(r_f(\gamma))(\tau)=\int_{-\frac dc}^{i \infty}i^{2-k} f(w) \left ( \frac{i-\tau}{w-\tau} \right )^{2-k}dw \qquad \text{if $\gamma=\begin{psmallmatrix}
			a & b\\
			c & d
			\end{psmallmatrix}$}
		$$
		is holomorphic provided that the integration path avoids the branch cut defined by $\frac{i-\tau}{w-\tau}$. In particular, it is holomorphic in a $-\frac dc$-excised neighbourhood $\Psi$ (see Figure \ref{fig-sp}).
		Importantly, the function is independent of the path of integration as long as the path lies inside the neighbourhood $V_{-\frac dc}(a,  \varepsilon)$ that we excise. In addition, it says that it has an asymptotic expansion as $\tau \to -\frac{d}{c}$.
		\item Assertion of Proposition \ref{BasTh} {\rm(2)} is the injectivity part of a potential Eichler--Shimura-type isomorphism for $S_k(\Gamma, \chi)$. In view of the isomorphisms proved in \cite[Theorems B, E]{BCD}, \cite[Theorem 1]{Kn}, and \cite[Theorem 2.1]{KM}, it is unlikely that the surjectivity holds as well. Indeed, if it would, then, with \eqref{eqcomm}, we would deduce that $\iota$ is an isomorphism between
		$H^1(\Gamma, D_{2-k}^{\omega, \infty, \text{exc}})$ and $H^1(\Gamma, D_{2-k}^{-\infty})$, which cannot be expected since $D_{2-k}^{-\infty}$ is much larger than $D_{2-k}^{\omega, \infty, \text{exc}}$.
	\end{enumerate}
\end{remarks}
\subsection{Depth $2$}

To describe the case of depth $2$, we give a more general definition that we need below, however, in this subsection, apply it only to the case $n=2$.
First, with the notation of Subsection \ref{DefMain}, we identify $\otimes_{j=1}^n \DD_{2-k_j} $ (resp. $\otimes_{j=1}^n \mathcal O (\mathbb H^-)$) with its image in $\DD_{(2-k_1)+\dots+(2-k_n)}$ (resp. $\mathcal O (\mathbb H^-)$) under the map induced by
\begin{equation}\label{ident}
a_1 \otimes \dots \otimes a_n \to a_1 \cdot \dots \cdot a_n.
\end{equation}

We denote by $|_{2-k_{1}, \dots, 2-k_{n}}$ the diagonal action of $\G$ on $\otimes_{j=1}^n \DD_{2-k_j}$ (resp. $\otimes_{j=1}^n \mathcal O (\mathbb H^-)$), i.e.,
\begin{equation}\label{|m}
	f|_{2-k_{1}, \dots, 2-k_{n}} \gamma(\tau):=\chi_{1}^{-1}(\gamma) \dots \chi_{n}^{-1}(\gamma)j(\gamma, \tau)^{k_{1}-2+\dots+k_{n}-2} f(\gamma \tau)
\end{equation}
for $f\in\otimes_{j=1}^n \DD_{2-k_j}$ (resp. $\otimes_{j=1}^n \mathcal O (\mathbb H^-)$). We next show the depth $2$ analogue of Lemma \ref{BKM23}.
Following \cite{BKM}, we work in the first instance just on $\mathbb H^-$. In the final section, we see that it can be extended to a larger excised neighborhood of $\mathbb H^- \cup \mathbb P^{1}(\mathbb R)$.
\begin{proposition}\label{2dim}
	For $f_j \in S_{k_j}(\Gamma,\chi_{j})$, $j\in\{1,2\}$ set, for $\tau \in \mathbb H^-,$
$$I_{f_1, f_2}(\tau) := \int_{\overline\tau}^{i \infty} \int_{w_1}^{i \infty}\frac{f_1(w_1)f_2(w_2)}{(w_1-\tau)^{2-k_1}(w_2-\tau)^{2-k_2}}dw_2 dw_1.$$
We then have, for $\gamma\in\Gamma$
\begin{equation}\label{fin0}
	I_{f_1, f_2}|_{2-k_1, 2-k_2}(\gamma-1) = -r_{f_1, f_2}(\gamma)+r_{f_1}(\gamma)r_{f_2}(\gamma)-r_{f_2}(\gamma) I_{f_1}.
\end{equation}
\end{proposition}
\begin{proof} By Theorem 5.1 of \cite{BKM}, we have that
\begin{equation*}
	I_{f_1, f_2}|_{2-k_1, 2-k_2}(\gamma-1)(\tau)
	=-\int_{\gamma^{-1}i \infty}^{i \infty} \int_{w_1}^{\gamma^{-1}i \infty}\frac{f_1(w_1)f_2(w_2)}{(w_1-\tau)^{2-k_1}(w_2-\tau)^{2-k_2}}dw_2 dw_1-r_{f_2}( \gamma)(\tau) I_{f_1}(\tau).
\end{equation*}
The double integral in the right-hand side can be written as
\begin{multline*}
	\left ( \int_{\gamma^{-1}i \infty}^{i \infty} \int_{w_1}^{i \infty}+
	\int_{\gamma^{-1}i \infty}^{i \infty} \int_{i \infty}^{\gamma^{-1}i \infty}\right ) \frac{f_1(w_1)f_2(w_2)}{(w_1-\tau)^{2-k_1}(w_2-\tau)^{2-k_2}}dw_2 dw_1 \\
	= r_{f_1, f_2}(\gamma)+ \int_{\gamma^{-1}i \infty}^{i \infty} \frac{f_1(w_1)}{(w_1-\tau)^{2-k_1}} dw_1 \, \int_{i \infty}^{\gamma^{-1}i \infty}\frac{f_2(w_2)}{(w_2-\tau)^{2-k_2}}dw_2,
\end{multline*}
which gives \eqref{fin0}.
\end{proof} Proposition
\ref{2dim} implies that $r_{f_1, f_2}$ satisfies \eqref{gen1coc} in way analogous to the way Lemma \ref{BKM23} implies the $1$-cocycle condition of $r_f$. The statement is generalised in the next section.
\begin{corollary}\label{2dimcor}
Let $k_1, k_2 \in \mathbb Z+\frac12$, $f_j \in S_{k_j}(\Gamma,\chi_{j})$ ($j\in\{1,2\}$), and $\gamma_1, \gamma_2 \in \Gamma$. We have
\begin{equation*}
	r_{f_1, f_2}(\gamma_{1}\gamma_{2})- r_{f_1, f_2}(\gamma_{1})|_{2-k_1, 2-k_2}\gamma_{2}- r_{f_1, f_2}(\gamma_{2}) =r_{f_1}(\gamma_{1})|_{k_1}\gamma_{2} \, \cdot r_{f_2}(\gamma_{2}).
\end{equation*}
\end{corollary}
\begin{proof}
Lemma \ref{BKM23} and \eqref{fin0} imply that
\noindent
\begin{align}\label{2dima}
&I_{f_1, f_2}|_{2-k_1, 2-k_2}(\gamma_{1}-1)|_{2-k_1, 2-k_2}\gamma_{2}\\
&=- r_{f_1, f_2}(\gamma_{1})|_{2-k_1, 2-k_2}\gamma_{2}+r_{f_1}(\gamma_1)|_{2-k_1}\gamma_{2} \, \cdot r_{f_2}(\gamma_{1})|_{2-k_2}\gamma_{2}-r_{f_2}(\gamma_{1})|_{2-k_2}\gamma_{2} \, \cdot I_{f_1}|_{2-k_1}\gamma_{2}\nonumber\\
&=- r_{f_1, f_2}(\gamma_{1})|_{2-k_1, 2-k_2}\gamma_{2}+r_{f_1}(\gamma_{1})|_{2-k_2}\gamma_{2} \, \cdot r_{f_2}(\gamma_{1})|_{2-k_1}\gamma_{2}-(r_{f_2}(\gamma_{1}\gamma_{2})-r_{f_2}(\gamma_{2})) \, \cdot I_{f_1}|_{2-k_1}\gamma_{2},\nonumber
\end{align}
where in the last line, we use \eqref{1cocy}. On the other hand, by Proposition \ref{2dim} with $\gamma_{1}\gamma_{2}$ and $\gamma_{2}$ instead of $\gamma_{1}$, we obtain
\begin{align}\label{2dimb}
&I_{f_1, f_2}|_{2-k_1, 2-k_2}(\gamma_{1}-1)\gamma_{2}=I_{f_1, f_2}|_{2-k_1, 2-k_2}(\gamma_{1}\gamma_{2}-1)-I_{f_1, f_2}|_{2-k_1, 2-k_2}(\gamma_{2}-1)\\
&=- r_{f_1, f_2}(\gamma_{1}\gamma_{2})+ r_{f_1, f_2}(\gamma_{2})+r_{f_1}(\gamma_{1}\gamma_{2}) \cdot r_{f_2}(\gamma_{1}\gamma_{2})-r_{f_1}(\gamma_{2}) r_{f_2}(\gamma_{2})-(r_{f_2}(\gamma_{1}\gamma_{2})-r_{f_2}(\gamma_{2})) I_{f_1}.\nonumber
\end{align}
Subtracting equation \eqref{2dima} from \eqref{2dimb}, we deduce that
\begin{multline*}
r_{f_1, f_2}(\gamma_{1}\gamma_{2})- r_{f_1, f_2}(\gamma_{1})|_{2-k_1, 2-k_2}\gamma_{2}- r_{f_1, f_2}(\gamma_{2})=r_{f_1}(\gamma_{1}\gamma_{2}) \cdot r_{f_2}(\gamma_{1}\gamma_{2})-r_{f_1}(\gamma_{2}) \cdot r_{f_2}(\gamma_{2})\\
-r_{f_1}(\gamma_{1})|_{2-k_2}\gamma_{2}\cdot r_{f_2}(\gamma_{1})|_{2-k_1}\gamma_{2}-(r_{f_2}(\gamma_{1}\gamma_{2})-r_{f_2}(\gamma_{2})) \cdot (I_{f_1}-I_{f_1}|_{2-k_1}\gamma_{2}).
\end{multline*}
Using \eqref{1cocy}, we derive Corollary \ref{2dimcor}.
\end{proof}

\section{Higher depths and a cohomological interpretation}\label{higher}
In this section we first prove the analogue of the cocycle condition for iterated integrals (both for integral and for half-integral weights). The form of that condition motivates the cohomological framework which we define and study in the following subsection.

\subsection{Higher depth period integrals}\label{higherper} In this subsection, we work jointly in integral and half-integral weights.
Let $\Gamma=\Gamma_0(N)$, $k_j \in \frac12\mathbb Z$, and $\chi_j$ a multiplier system of weight $k_j$ ($j\in\{1, \dots, n\}$). If $k_j \in \mathbb Z+\frac12$, then $4|N$ and, if $k_j \in 2\mathbb N$, then $\chi_j=1$. For $f_j \in S_{k_j}(\Gamma,\chi_j)$, we consider the \emph{period integrals of depth $n$}
\begin{equation*}
r_{f_1, \dots f_n}(\gamma)(\tau):=r^*_{f_1, \dots f_n}(k_1-2, \dots, k_n-2)(\gamma)(\tau)
=\int_{\gamma^{-1}i\infty}^{i \infty} \int_{w_1}^{i \infty} \dots \int_{w_{n-1}}^{i \infty} \prod_{j=1}^n\frac{f_j(w_j)}{(w_j-\tau)^{2-k_j}} dw_n \dots dw_1
\end{equation*}
initially defined in $\mathbb H^-$. The following theorem generalizes Theorem \ref{2dimZ} and Corollary \ref{2dimcor}.
\begin{theorem}\label{propfinn}
	For $0\le j\le n-2$, there exist $\sigma_j: \Gamma_0(N) \to \mathcal O(\mathbb H^-)$, such that, for each $\gamma_{1}, \gamma_{2} \in \Gamma_0(N)$, we have
	\begin{equation*}
		r_{f_1, \dots, f_n}(\gamma_{1}\gamma_{2})-r_{f_1, \dots, f_n}( \gamma_{1})|_{2-k_1, \dots, 2-k_n}\gamma_{2}- r_{f_1, \dots, f_n}(\gamma_{2}) = \sum_{j=1}^{n-1} r_{f_1, \dots, f_{j}}(\gamma_{1})|_{2-k_1, \dots, 2-k_{j}}\gamma_{2} \cdot  \sigma_j(\gamma_{2}).
	\end{equation*}
The map $\sigma_j(\gamma)$ can be expressed explicitly as
\begin{equation*}
\sigma_j(\gamma)=\sum_{\ell \in \mathbb N} \sum_{(I_1, I_2, \dots I_\ell)} (-1)^\ell r_{I_1}\left(\gamma^{-1}\right)|_{2-I_1}\gamma \cdot r_{I_2}\left(\gamma^{-1}\right)|_{2-I_2}\gamma  \dots r_{I_\ell}\left(\gamma^{-1}\right)|_{2-I_\ell}\gamma,
	\end{equation*}
where $(I_1, I_2, \dots I_\ell)$ is a vector of sub-vectors partitioning $(j+1, \dots, n)$ and
$|_{2-I_m}$ stands for $|_{2-k_{j_1}, 2-k_{j_2}, \dots}$, if $I_m=(j_1, j_2, \dots).$
\end{theorem}
\begin{proof}
	We prove the claim by induction on $n$. For $n=2$, the statement is Corollary \ref{2dimcor}.  A change of variables in the integral in terms of $w_n$ combined with \eqref{modul}, imply
	\begin{multline*}
		r_{f_1, \dots, f_n}(\gamma_{1})|_{2-k_1, \dots, 2-k_n}\gamma_{2}(\tau) = \int_{\gamma_{1}^{-1}i \infty}^{i \infty} \int_{w_1}^{i \infty} \dots \int_{w_{n-1}}^{i \infty} \prod_{j=1}^n\frac{\chi_j^{-1}(\gamma_{2}) f_j(w_j)}{(w_j-\gamma_{2}\tau)^{2-k_j}(\gamma_2, \tau)^{2-k_j}} dw_n \dots dw_1  \\
		=\int_{\gamma_1^{-1}i \infty}^{i \infty} \int_{w_1}^{i \infty} \dots \int_{w_{n-2}}^{i \infty} \prod_{j=1}^{n-1}\frac{\chi_j^{-1}(\gamma_2) f_j(w_j)}{(w_j-\gamma_{2}\tau)^{2-k_j}(\gamma_2, \tau)^{2-k_j}}\int_{\gamma_{2}^{-1}w_{n-1}}^{\gamma_2^{-1} i \infty} \frac{f_n(w_n)}{(w_n-\tau)^{2-k_n}}dw_n \dots dw_1.
	\end{multline*}
	Iteratively the changes of variables $w_j \mapsto \gamma_2^{-1}w_j$ as $j$ ranges from $n-1$ to $1$, we deduce that
	\begin{equation}\label{|N}
		r_{f_1, \dots, f_n}(\gamma_{1})|_{2-k_1, \dots, 2-k_n}\gamma_{2}(\tau) = \int_{(\gamma_1 \gamma_{2})^{-1}i \infty}^{\gamma_2^{-1}i \infty} \int_{w_1}^{\gamma_{2}^{-1}i \infty} \dots \int_{w_{n-1}}^{\gamma_{2}^{-1}i \infty} F(\bm{w})\bm{dw},
	\end{equation}
	where
	\begin{equation*}
		\bm{dw}:=dw_{n}\dots dw_{1}, \quad F(\bm{w}):=\prod_{j=1}^n\frac{ f_j(w_j)}{(w_j-\tau)^{2-k_j}}.
	\end{equation*}
	By splitting the second integral in \eqref{|N} we see
	\begin{align}
		&\hspace{-1cm}r_{f_1, \dots, f_n}(\gamma_{1})|_{2-k_1, \dots, 2-k_n}\gamma_{2} (\tau) \nonumber\\
		&= \left ( \int_{(\gamma_{1}\gamma_{2})^{-1}i \infty}^{\gamma_{2}^{-1}i \infty} \int_{w_1}^{i \infty}
		\int_{w_{2}}^{\gamma_{2}^{-1}i \infty}  \dots \int_{w_{n-1}}^{\gamma_2^{-1}i \infty}
		+\int_{(\gamma_{1}\gamma_{2})^{-1}i \infty}^{\gamma_{2}^{-1}i \infty} \int_{i \infty}^{\gamma_{2}^{-1}i \infty} \dots \int_{w_{n-1}}^{\gamma_{2}^{-1}i \infty} \right ) F(\bm{w})\bm{dw}\nonumber\\
		&= \int_{(\gamma_{1}\gamma_{2})^{-1}i \infty}^{\gamma_{2}^{-1}i \infty} \int_{w_1}^{i \infty}
		\int_{w_{2}}^{\gamma_{2}^{-1}i \infty} \dots \int_{w_{n-1}}^{\gamma_{2}^{-1}i \infty} F(\bm{w})\bm{dw} \nonumber\\
&\qquad \qquad  \qquad \qquad +\left(r_{f_1}(\gamma_{1}\gamma_{2})(\tau)-r_{f_1}(\gamma_{2})(\tau)\right)
r_{f_2, \dots, f_n}\left(\gamma_2^{-1}\right)|_{2-k_1, \dots 2-k_n}\gamma_2(\tau)
,\label{w1}
	\end{align}
since, by \eqref{|N},
	$$r_{f_m, \dots, f_n}\left(\gamma^{-1}\right)|_{2-k_m, \dots 2-k_n}\gamma (\tau)=\int_{i \infty}^{\gamma^{-1}i \infty} \dots \int_{w_{n-1}}^{\gamma^{-1}i \infty} F(\bm{w})\bm{dw}.$$
	The multiple integral on the right-hand side equals
	\begin{multline*}
		\left ( \int_{(\gamma_{1}\gamma_2)^{-1}i \infty}^{\gamma_{2}^{-1}i \infty} \int_{w_1}^{i \infty} \int_{w_2}^{i \infty} \dots \int_{w_{n-1}}^{\gamma_{2}^{-1}i \infty}
		+\int_{(\gamma_{1}\gamma_{2})^{-1}i \infty}^{\gamma_{2}^{-1}i \infty} \int_{w_1}^{i \infty} \int_{i \infty}^{\gamma_{2}^{-1}i \infty} \dots \int_{w_{n-1}}^{\gamma_{2}^{-1}i \infty} \right ) F(\bm{w})\bm{dw} \\
		= \int_{(\gamma_{1}\gamma_{2})^{-1}i \infty}^{\gamma_{2}^{-1}i \infty} \int_{w_1}^{i \infty} \int_{w_2}^{i \infty} \int_{w_{3}}^{\gamma_2^{-1}i \infty}  \dots \int_{w_{n-1}}^{\gamma_{2}^{-1}i \infty}F(\bm{w})\bm{dw}\\
\qquad \qquad \qquad \qquad +\left(r_{f_1, f_2}(\gamma_{1}\gamma_{2})(\tau)-r_{f_1, f_2}(\gamma_{2})(\tau)\right) r_{f_3, \dots, f_n}\left(\gamma_2^{-1}\right)|_{2-k_3, \dots 2-k_n}\gamma_2(\tau).
	\end{multline*}
Plugging into \eqref{w1}, we obtain
	\begin{multline*}
		\int_{(\gamma_{1}\gamma_{2})^{-1}i \infty}^{\gamma_{2}^{-1}i \infty} \int_{w_1}^{i \infty} \int_{w_2}^{i \infty} \int_{w_{3}}^{\gamma_2^{-1}i \infty}  \dots \int_{w_{n-1}}^{\gamma_{2}^{-1}i \infty}F(\bm{w})\bm{dw}
\\
+\sum_{j=1}^2\left(r_{f_1, \dots, f_{j}}(\gamma_{1}\gamma_{2})(\tau)-r_{f_1, \dots, f_{j}}(\gamma_{2})(\tau)\right)r_{f_{j+1}, \dots, f_n}\left(\gamma_2^{-1}\right)|_{2-k_{j+1}, \dots 2-k_n}\gamma_2(\tau).
	\end{multline*}
	Continuing in this way we deduce, using the inductive hypothesis,
	\begin{align*}
		r_{f_1, \dots, f_n}(\gamma_{1})&|_{2-k_1, \dots, 2-k_n}\gamma_{2}(\tau)
		=\int_{(\gamma_{1}\gamma_{2})^{-1}i \infty}^{\gamma_{2}^{-1}i \infty} \int_{w_1}^{i \infty}  \dots \int_{w_{n-1}}^{i \infty} F(\bm{w})\bm{dw}\\
&+ \sum_{j=1}^{n-1}\left(r_{f_1, \dots, f_{j}}(\gamma_{1}\gamma_{2})-r_{f_1, \dots, f_{j}}(\gamma_{2})\right)
r_{f_{j+1}, \dots, f_n}\left(\gamma_2^{-1}\right)|_{2-k_{j+1}, \dots 2-k_n}\gamma_2(\tau)
\\
		&=r_{f_1, \dots, f_n}(\gamma_{1}\gamma_{2})-r_{f_1, \dots, f_n}(\gamma_{2})-\sum_{j=1}^{n-1}r_{f_1, \dots, f_{j}}(\gamma_{1})|_{2-k_1, \dots, 2-k_j}\gamma_{2} \, \cdot \sigma_j(\gamma_2)(\tau).\qedhere
	\end{align*}
\end{proof}
\subsection{Cohomological interpretation}\label{coh}
Theorem \ref{propfinn} motivates the following inductive definition.
\begin{definition}\label{mainDef}
	For $j\in\N$, let $M_{j}$ be right $G$-modules. We set $Z_{(0)}^1(G, 0)=L_{0}^m=0$ for $m \in \mathbb N_0$.
Suppose that, for all non-negative integers $r < k$, a space $L_{(r)}^m \subset\oplus_{j=1}^{r-1} ( C^1(G, \otimes_{\ell=1}^j M_\ell) \otimes C^{m-1}(G, \otimes_{\ell=j+1}^r M_\ell) )$ and a space $Z^1_{(r)}(G, \otimes_{j=1}^r M_j) \subset C^1\left(G, \otimes_{\ell=1}^{r}M_{\ell} \right)\big/\mu(L_{(r)}^1)$, where $\mu$ is the map \eqref{mu}, are defined for all $m \in \mathbb N_0$. Then, for $r < k$, let
\begin{equation}\label{natpr}\pi_{(r)}:C^1\left(G, \otimes_{\ell=1}^{r}M_{\ell}\right) \to C^1\left(G, \otimes_{\ell=1}^{r}M_{\ell} \right)\Big/\mu\left(L_{(r)}^1\right)
\end{equation}
be the natural projection. Set
$$ L_{(k)}^{m}:=\bigoplus_{j=1}^{k-1} \left( \pi_{(j)}^{-1}Z_{(j)}^1\left(G, \otimes_{\ell=1}^jM_\ell\right) \otimes C^{m-1}\left(G, \otimes_{\ell=j+1}^{k} M_{\ell}\right)\right).$$
Let
\begin{equation*}
d_{(k)}^m: C^{m}\left(G, \otimes_{j=1}^{k} M_j\right) \Big/ \mu\left(L_{(k)}^{m}\right) \longrightarrow C^{m+1}\left(G, \otimes_{j=1}^{k} M_j\right) \Big/ \mu\left(L^{m+1}_{(k)}\right)
\end{equation*}
be the map induced by $d^m: C^{m}(G, \otimes_{j=1}^{k} M_j)   \rightarrow C^{m+1}(G, \otimes_{j=1}^{k} M_j).$
Furthermore,
\begin{align*}
	&Z_{(k)}^m(G; \otimes_{j=1}^k M_j)=\operatorname{ker}\left( d^m_{(k)}\right ), \qquad B_{(k)}^m(G; \otimes_{j=1}^k M_j)=\operatorname{im}\left( d^{m-1}_{(k)}\right ) , \\
	&H_{(k)}^m\left(G; \otimes_{j=1}^k M_j\right)=Z_{(k)}^m\left(G; \otimes_{j=1}^k M_j\right)\Big/B_{(k)}^m\left(G; \otimes_{j=1}^k M_j\right).
\end{align*}
We call $Z_{(k)}^m$ the {\it group of $m$-th cocycles of depth $k$}. Likewise, for the depth $k$ coboundary and cohomology groups $Z_{(k)}^m$ and $H_{(k)}^m$, respectively.
\end{definition}
To show that this is well-defined we use the following lemma.
\begin{lemma}\label{comm} With the notation of Definition \ref{mainDef} we have, for $k, m \in \N_{0}$,
$$d^{m}\left(\mu\left(L_{(k)}^m\right)\right) \subset \mu\left(L_{(k)}^{m+1}\right).$$
In particular, the map $d_{(k)}^m$ is well-defined.
\end{lemma}
\begin{proof}
	We prove the claim by induction on $k$. It is trivial for $k=0$. Assume next that the claim holds for all $k<k_0$. To complete the proof, it suffices to show that, for all $1\le j\le k_0-1$, and $\sigma_1 \in \pi_j^{-1}Z_{(j)}^1(G, \otimes_{\nu=1}^j M_\nu)$, $\sigma_2 \in C^{m-1}(G, \otimes_{\nu=j+1}^{k_0}M_{\nu})$, we have
\begin{equation}\label{w-d}d^{m}(\mu(\sigma_1 \otimes \sigma_2)) \in \mu\left(L_{(k_0)}^{m+1}\right).\end{equation}
By the ``Leibniz rule" for cup products (see, e.g. \cite{Br}, page 111), we have
$$d^{m}(\mu(\sigma_1 \otimes \sigma_2))=d^{m}(\sigma_1 \cup \sigma_2)=
d^{1}(\sigma_1) \cup \sigma_2-\sigma_1 \cup d^{m-1}\sigma_2.$$
By the inductive hypothesis, $d_{(j)}^1$ is well-defined and thus, with the definition of $Z_{(j)}^1(G, \otimes _{\nu=1}^j M_\nu)$, $d^1\sigma_1 \in \mu(L_{(j)}^2)$, i.e., $d^1\sigma_1=\sum_{\ell=1}^{j-1} \phi_\ell \cup \varrho_{\ell}$ for some $\phi_\ell \in \pi_{(\ell)}^{-1}Z_{(\ell)}^1(G, \otimes_{r=1}^\ell M_r), \varrho_{\ell} \in C^1(G, \otimes_{r=\ell+1}^j M_{r})$.
Then we have that each $(\phi_\ell \cup \varrho_{\ell}) \cup \sigma_2=\phi_\ell \cup \psi_{\ell}$, for $\psi_{\ell}:=\varrho_{\ell} \cup \sigma_2 \in C^{m}(G, \otimes_{r=\ell+1}^{k_0}M_r)$, and $\sigma_1 \cup d^{m-1}\sigma_2 \in \pi_j^{-1}Z_{(j)}^1(G, \otimes_{r=1}^j M_{r}) \cup C^m(G, \otimes_{r=j+1}^{k_0} M_r$), i.e., \eqref{w-d} holds.
\end{proof}

We illustrate the nature of Definition \ref{mainDef} by writing down explicitly the defining relations of $m$-cocycles of depth $k$ and $m$-coboundaries in the case of low $k$ and $m$.

\subsubsection{Low depths.} \label{ld}\hspace{0cm}
\begin{itemize}[leftmargin=*,labelindent=0pt]
	\item $k=1$: We trivially have $L_{(1)}^m=0$ and thus $d_{(1)}^m=d^m$ (the standard differential \eqref{differential}). Therefore,
	$Z_{(1)}^{m}(G, M_1)=\operatorname{ker}(d^m)=Z^m(G, M_1),$
	$B_{(1)}^{m}(G, M_1)=\operatorname{im}(d^{m-1})=B^m(G, M_1)$, and
	$H_{(1)}^{m}(G, M_1)=H^m(G, M_1).$
	\item $k=2$: Since, trivially, $L_{(1)}^m=0$, the projection $\pi_{(1)}$ is the identity map. Hence $\pi_{(1)}^{-1}Z_{(1)}^1(G, M_1)=Z^1(G, M_1)$ and therefore, $L_{(2)}^m=\pi_{(1)}^{-1}Z_{(1)}^1(G, M_1) \otimes C^{m-1}(G, M_2)=Z^1(G, M_1) \otimes C^{m-1}(G, M_2)$.
	This implies that
	\begin{equation*}
	d_{(2)}^m: C^{m}\left(G, M_1 \otimes M_2 \right) \big/ \mu\left(L_{(2)}^m\right) \longrightarrow C^{m+1}\left(G, M_1 \otimes M_2\right)\big/ \mu\left(L_{(2)}^{m+1} \right)
	\end{equation*}
	and
	$Z_{(2)}^{m}(G, M_1 \otimes M_2)=\operatorname{ker}(d_{(2)}^m)$,
	$B_{(2)}^{m}(G, M_1 \otimes M_2)=\operatorname{im}(d_{(2)}^{m-1})$, and $H_{(2)}^{m}(G, M_1 \otimes M_2)=Z_{(2)}^{m}(G, M_1 \otimes M_2)/B_{(2)}^{m}(G, M_1 \otimes M_2).$
\end{itemize}
\subsubsection{Low orders}\label{lo}\hspace{0cm}
\begin{itemize}[leftmargin=*,labelindent=0pt]
	\item $m=0$: Since $C^{m}=0$ if $m<0$, we have $L_{(k)}^0=0$ and thus $H^0_{(k)}(G, \otimes_{j=1}^k M_j)$ consists of
    $a \in \otimes_{j=1}^kM_j$ such that $$a \circ (g-1)=\sum_{j=1}^{k-1}\sum_{s} \varrho_{j, s}(g) \otimes \phi_{j, s}  \, \, \text{for all $g \in G$}$$
    for some $\varrho_{j, s} \in \pi_{(j)}^{-1}Z_{(j)}^1(G, \otimes_{n=1}^j M_n)$ and $\phi_{j, s} \in \otimes_{n=j+1}^{k} M_n$.
	\item $m=1$:
	The space $Z^1_{(k)}(G; \otimes_{j=1}^k M_j)$ consists of classes $[\sigma] \in C^{1}(G, \otimes_{j=1}^kM_j ) / \mu(L_{(k)}^1)$ such that $\sigma: G \to \otimes_{j=1}^kM_j$ satisfies
	\begin{equation}\label{Z1}
	\sigma(g_2 g_1)-\sigma(g_2) \circ g_1-\sigma(g_1)=	\sum_{j=1}^{k-1}\sum_{s} \varrho_{j, s}(g_2)\circ g_1 \otimes \phi_{j, s} (g_1)  \, \, \text{for all $g_1, g_2 \in G$}
	\end{equation}
	for some $\varrho_{j, s} \in \pi_{(j)}^{-1}Z_{(j)}^1(G, \otimes_{n=1}^j M_n)$ and $\phi_{j, s}: G \to \otimes_{n=j+1}^{k} M_n$.
\end{itemize}

\begin{remarks}\hspace{0cm}
	\begin{enumerate}[leftmargin=*,label=\rm(\arabic*)]
		\item Informally, \eqref{Z1} means that $\sigma$ satisfies the classical $1$-cocycle relation, up to cup products involving \emph{lower order} cocycles.
		\item Caution is called for regarding the actions in \eqref{Z1}: The first ``\,$\circ$" denotes the action of $G$ on $\otimes_{j=1}^kM_j$, whereas the second one is the action of $G$ on
		$\otimes_{n=1}^jM_n$. We do not introduce separate notation for each action to avoid making the notation unmanageable.
	\end{enumerate}
\end{remarks}

Furthermore, $B^1_{(k)}(G; \otimes_{j=1}^k M_j)$ contains the classes $[\sigma] \in C^1(G, \otimes_{j=1}^kM_j ) / \mu(L_{(k)}^1)$ such that $\sigma: G^1 \to \otimes_{j=1}^kM_j$ satisfies, for all $g \in G,$
\begin{equation*}
\sigma(g)=\phi \circ (g-1)+
\sum_{j=1}^{k-1} \sum_r \sigma_{j, r}(g) \otimes \varrho_{j, r} ,
\end{equation*} for some $\phi \in \otimes_{j=1}^{k} M_j$ and some
$$\sum_r \sigma_{j, r} \otimes \varrho_{j, r} \in \pi_{(j)}^{-1}Z_{(j)}^1\left(G, \otimes_{n=1}^jM_n\right) \otimes C^{0}\left(G, \otimes_{n=j+1}^k M_j\right), \, \, 1\leq j \leq k.$$

More generally, $Z_{(k)}^m(G; \otimes_{j=1}^k M_j)$ consists of classes $[\sigma] \in C^m(G, \otimes_{j=1}^kM_j ) / \mu(L_{(k)}^m)$ such that $\sigma: G^m \to \otimes_{j=1}^kM_j$ satisfies
\begin{equation}\label{Zn}d^m \sigma= \sum_{j=1}^{k-1} \sum_r \mu \left ( \sigma_{j, r} \otimes \varrho_{j, r} \right ) ,
\end{equation} for some
$$\sum_r \sigma_{j, r} \otimes \varrho_{j, r} \in \pi_{(j)}^{-1} Z^1_{(j)}\left(G, \otimes_{n=1}^{j}M_{n}\right) \otimes C^{m}\left(G, \otimes_{n=j+1}^{k} M_{n}\right), \quad \,1\le j\le k-1.$$

For $m \in \N$, $B_{(k)}^m(G; \otimes_{j=1}^k M_j)$ contains the classes $[\sigma] \in C^m(G, \otimes_{j=1}^kM_j ) / \mu(L_{(k)}^m)$ such that $\sigma: G^m \to \otimes_{j=1}^kM_j$ satisfies
\begin{equation}\label{Bn} \sigma=d^{m-1}\phi+\sum_{j=1}^{k-1} \sum_r \mu \left ( \sigma_{j, r} \otimes \varrho_{j, r} \right ) ,
\end{equation} for some $\phi \in C^{m-1}(G, \otimes_{j=1}^{k} M_j)$ and
$$\sum_r \sigma_{j, r} \otimes \varrho_{j, r} \in \pi_{(j)}^{-1}Z^1_{(j)}\left(G, \otimes_{n=1}^jM_n\right) \otimes C^{m-1}\left(G, \otimes_{n=j+1}^k M_n\right), \, \, 1\leq j\leq k-1.$$

From the above explicit relations, we can view the depth $k$ $1$-cocycle as a mapping that satisfies the standard $1$-cocycle condition up to maps that decompose, in a certain way, into lower depth mappings. As with the classical cohomology, for the kind of groups we are considering, our cohomology is concentrated in orders $0$ and $1$. To show this, we need the following lemma.
\begin{lemma}\label{compl} Let $M_j$ ($1\le j\le k$) be right $G$-modules. Then, for each $[\sigma] \in Z_{(k)}^{n}(G, \otimes_{j=1}^k M_j)$, there exist $\sigma_{j, r} \in \pi_{(j)}^{-1}Z_{(j)}^1(G, \otimes_{\nu=1}^j M_\nu)$, and $\varrho_{j, r} \in
\pi_{(k-j)}^{-1}Z_{(k-j)}^n(G, \otimes_{\nu=j+1}^k M_\nu)$ such that
\begin{equation}\label{Znlem}d^n \sigma= \sum_{j=1}^{k-1} \sum_r  \sigma_{j, r} \cup \varrho_{j, r}.
\end{equation}
\end{lemma}
\begin{proof} We proceed by induction on $k$. Let $k=1$. Since $Z^n_{(1)}=Z^n$, we have $d^n\sigma=0$ and the assertion is trivial. Suppose it holds for depths $<k$. We show that the claim holds for depth $k$.

Fix a $\mathbb C$-basis $\{\sigma_{j, r}\}_{j, r}$ of $\bigoplus_{j=1}^{k}\pi_{(j)}^{-1}Z_{(j)}^1(G, \otimes_{\nu=1}^j M_{\nu})$. Let $[\sigma] \in Z_{(k)}^n(G, \otimes_{j=1}^k M_j)$. By \eqref{Zn}, there exist $\varrho_{j, r} \in
C^n(G, \otimes_{\nu=j+1}^{k} M_{\nu})$ such that $d^n \sigma= \sum_{{\substack{j=1 \\ r}}}^{k-1}  \sigma_{j, r} \cup \varrho_{j, r}$, where $r$ ranges over the indices of the basis elements $\sigma_{j, r}$. By the Leibniz rule we have
\begin{equation}\label{Leib}0=d^{n+1}(d^n \sigma)=\sum_{{\substack{j=1 \\ r}}}^{k-1} \left( d^1 \sigma_{j, r} \cup \varrho_{j, r} \right) -\sum_{{\substack{j=1 \\ r}}}^{k-1} \left( \sigma_{j, r} \cup d^n \varrho_{j, r} \right).
\end{equation}
The inductive hypothesis implies that there exist $\psi^{j, r}_{\nu, \ell} \in \pi_{(j-\nu)}^{-1}Z_{(j-\nu)}^1(G, \otimes_{m=\nu+1}^j M_m)$ such that the first double sum becomes
\begin{align*}
\sum_{{\substack{j=1 \\ r}}}^{k-1} \left ( \sum_{{\substack{\nu=1 \\ \ell}}}^{j-1} \sigma_{\nu, \ell} \cup \psi^{j, r}_{\nu, \ell} \right ) \cup \varrho_{j, r}
=\sum_{{\substack{\nu=1\\ \ell }}}^{k-2} \sum_{{\substack{j=\nu+1 \\ r}}}^{k-1} \sigma_{\nu, \ell} \cup \psi^{j, r}_{\nu, \ell} \cup \varrho_{j, r}
=\sum_{{\substack{\nu=1 \\ \ell}}}^{k- 1} \sigma_{\nu, \ell} \cup \left ( \sum_{{\substack{j=1 \\ r}}}^{k-\nu-1}   \psi^{j+\nu, r}_{\nu, \ell} \cup \varrho_{j+\nu, r} \right ).
\end{align*}
 Note that the inner sum vanishes for $\nu=k-1$.  Applying to \eqref{Leib}, we get
\begin{equation}\label{linind}0=\sum_{{\substack{j=1 \\ r}}}^{k-1} \sigma_{j, r} \cup \left (\sum_{{\substack{\nu=1 \\ \ell}}}^{k-j-1} \psi^{j+\nu, \ell}_{j, r}  \cup \varrho_{\nu+j, \ell}  -d^n \varrho_{j, r} \right ).
\end{equation}
If $\phi_{j, r}$ denotes the term inside the parentheses in \eqref{linind}, we deduce that, for all $\gamma_1, \gamma_2 \in G$, $\sum_{j,  r} \sigma_{j, r}(\gamma_2) \circ \gamma_1 \otimes \phi_{j, r}(\gamma_1)=0$, and thus
 $\sum_{j,  r} \sigma_{j, r}(\gamma_2) \otimes \phi_{j, r}(\gamma_1) \circ \gamma_1^{-1}=0$.
 Since $\sigma_{j, r}$ are linearly independent maps of $G$, we deduce that $\phi_{j, r}=0$ for all $j, r$ and hence
$$d^n \varrho_{j, r}=\sum_{{\substack{\nu=1 \\ \ell}}}^{k-j-1} \psi^{j+\nu, \ell}_{j, r} \cup \varrho_{\nu+j, \ell}.$$ Thus $\varrho_{j, r} \in \pi_{(k-j)}^{-1}Z_{(k-j)}^n(G, \otimes_{\nu=j+1}^{k} M_\nu)$, proving \eqref{Znlem}.
\end{proof}
We can now prove the following proposition.
\begin{proposition}\label{cdim1} Let $\Gamma\in\{\Gamma_0(N),\langle \Gamma_0(N),W_N \rangle\}$ and let $\{M_j\}_{j=1}^k$ be $\mathbb C$-vector spaces on which $\Gamma$ acts. For $n \ge 2$, we have
$$H^n_{(k)}\left(\Gamma; \otimes_{j=1}^k M_j\right)=0.$$
\end{proposition}
\begin{proof}
We prove the proposition using induction on $k$. For $k=1$, the claim becomes Lemma 7.2 of \cite{DO} because $H^n_{(1)}(\Gamma; M_1)=H^n(\Gamma, M_1)$. Let $k \ge 2$ and assume that the claim holds for depths $<k$.
In the sequel, since we work with a fixed group $\Gamma$, to simplify notation, we omit $\Gamma$ from the notation, e.g., we write  $Z_{(k)}^{n}(\otimes_{j=1}^k M_j)$ instead of $Z_{(k)}^{n}(\Gamma, \otimes_{j=1}^k M_j)$.

Let $[\sigma] \in Z_{(k)}^n(\otimes_{j=1}^k M_j)$. By Lemma \ref{compl},
there exist $\sigma_{j, r} \in \pi_{(j)}^{-1}Z_{(j)}^1(\otimes_{\nu=1}^j M_\nu)$, $\varrho_{j, r} \in \pi_{(k-j)}^{-1}Z_{(k-j)}^n$ $(\otimes_{\nu=j+1}^k M_\nu)$ such that
\begin{equation}\label{Zn'}d^n \sigma= \sum_{{\substack{j=1 \\ r}}}^{k-1} \sigma_{j, r} \cup \varrho_{j, r}.
\end{equation}
The inductive hypothesis implies that $[\varrho_{j, r}] \in B_{(k-j)}^n(\otimes_{\nu=j+1}^k M_\nu)$, or, by \eqref{Bn},
\begin{equation}\label{Bn'} \varrho_{j, r}=d^{n-1}\chi_{j, r}+\sum_{{\substack{\nu=1 \\ s}}}^{k-j-1} \phi_{\nu, s} \cup \psi_{\nu, s},
\end{equation} for some
$\phi_{\nu, s} \in \pi_{(\nu)}^{-1}Z^1_{(\nu)}(\otimes_{m=j+1}^{j+\nu} M_m)$,  $\psi_{\nu, s}  \in C^{n-1} (\otimes_{m=j+\nu+1}^{k} M_m)$ and $\chi_{j, r} \in C^{n-1}(\otimes_{\nu=j+1}^{k} M_\nu)$.
Because of \eqref{Bn'} and the Leibniz rule, \eqref{Zn'} implies
\begin{equation*}
d^n\left ( \sigma+\sum_{{\substack{j=1 \\ r }}}^{k-1}  \sigma_{j, r} \cup \chi_{j, r} \right )=\sum_{{\substack{j= 2  \\ r}}}^{k-1}  d^1\sigma_{j, r} \cup \chi_{j, r}+\sum_{{\substack{j=1 \\ r}}}^{k-2}  \sigma_{j, r} \cup \left ( \sum_{{\substack{\nu=1 \\ s}}}^{k-j-1} \phi_{\nu, s} \cup \psi_{\nu, s} \right ).
\end{equation*}
 The lower limit of the first sum in the right-hand side is $2$ because $d^1\sigma_{1, r}=0$ (since $\sigma_{1, r} \in Z_{(1)}^1=Z^1$) and the upper limit of the second sum is $k-2$ because the inner sum is $0$ if $j=k-1$.
By \eqref{Zn},
\begin{align}\label{d1}\sum_{{\substack{j= 2  \\ r}}}^{k-1}  &d^1\sigma_{j, r} \cup \chi_{j, r}  = \sum_{{\substack{j=1 \\ r}}}^{k-2}  d^1\sigma_{j+1, r} \cup \chi_{j+1, r} \\
 &\in \sum_{{\substack{j=1 \\ r}}}^{k-2} \sum_{{\substack{\nu=1 \\ s}}}^{j}  \pi_{(\nu)}^{-1}Z_{(\nu)}^1\left (\bigotimes_{\ell=1}^{\nu} M_\ell \right ) \cup C^{1}\left (\bigotimes_{\ell=\nu+1}^{j+1} M_\ell \right )
\cup C^{n-1}\left (\bigotimes_{\ell=j+2}^k M_\ell \right ) \nonumber\\
&\subset
\sum_{{\substack{j=1 \\ r}}}^{k-2}   \sum_{{\substack{\nu=1 \\ s}}}^{j}  \pi_{(\nu)}^{-1}Z_{(\nu)}^1\left (\bigotimes_{\ell=1}^{\nu} M_\ell \right )
\cup C^{n}\left (\bigotimes_{\ell=\nu+1}^k M_\ell \right ) \subset
\sum_{{\substack{\nu=1 \\ s}}}^{k-2} \pi_{(\nu)}^{-1}Z_{(\nu)}^1 \left (\bigotimes_{\ell=1}^{\nu} M_\ell \right )
\cup C^{n} \left (\bigotimes_{\ell=\nu+1}^k M_\ell \right ). \nonumber
\end{align}
We let $\mathcal{M}_{\ell}:=M_{\ell}$, for $\ell\in\{1, \dots, k-2\}$ and $\mathcal{M}_{k-1}:=M_{k-1} \otimes M_k$. Then, by \eqref{d1} and \eqref{Zn}, we deduce that
$\sigma+\sum_{{\substack{j=1}}}^{k-1} \sigma_{j, r} \cup \chi_{j, r}\in\pi_{(k-1)}^{-1}Z_{(k-1)}^n(\otimes_{j=1}^{k-1}  \mathcal{M}_j )$, which, by inductive hypothesis is $\pi_{(k-1)}^{-1}B_{(k-1)}^n(\otimes_{j=1}^{k-1}  \mathcal{M}_j  )$. Therefore, by \eqref{Bn}, for some $\psi \in C^{n-1}(\otimes_{j=1}^{k  -1 }  \mathcal{M}_j  )$,
$$\sigma+\sum_{{\substack{j=1 \\ r}}}^{k-1} \sigma_{j, r} \cup \chi_{j, r} \in d^{n-1}\psi+  \sum_{{\substack{j=1}}}^{k-2}  \pi_{(j)}^{-1}Z_{(j)}^1\left (\bigotimes_{\nu=1}^{j} \mathcal{M}_\nu  \right )
\cup C^{n-1}\left (\bigotimes_{\nu=j+1}^{k-1}  \mathcal{M}_\nu  \right ).$$
 Since the sum in the left-hand side belongs to $$\sum_{{\substack{j=1}}}^{k-2}  \pi_{(j)}^{-1}Z_{(j)}^1 \lp\bigotimes_{\nu=1}^{j} \mathcal{M}_\nu  \rp
\cup C^{n-1} \lp \bigotimes_{\nu=j+1}^{k  -1 } \mathcal{M}_{\nu} \rp+
\pi_{(k-1)}^{-1}Z_{(k-1)}^1 \lp\bigotimes_{\nu=1}^{k-1} M_\nu  \rp
\cup C^{n-1}(M_{k})
,$$ 
we deduce that
\begin{multline*}
\sigma \in d^{n-1}\psi+\sum_{{\substack{j=1}}}^{k-2} \pi_{(j)}^{-1}Z_{(j)}^1\left (\bigotimes_{\nu=1}^{j} \mathcal{M}_\nu \right )
\cup C^{n-1}\left (\bigotimes_{\nu=j+1}^{k-1} \mathcal{M}_{\nu} \right )+
\pi_{(k-1)}^{-1}Z_{(k-1)}^1 \lp\bigotimes_{\nu=1}^{k-1} M_\nu  \rp
\cup C^{n-1}(M_{k})
\\=d^{n-1}\psi+
\sum_{{\substack{j=1}}}^{k-1} \pi_{(j)}^{-1}Z_{(j)}^1\left (\bigotimes_{\nu=1}^{j} M_\nu \right )
\cup C^{n-1}\left (\bigotimes_{\nu=j+1}^{k} M_{\nu} \right ).  
\end{multline*}
Hence \eqref{Bn} implies that $[\sigma]  \in B_{(k)}^n(\otimes_{j=1}^k M_j)
$.
\end{proof}

We now apply the general construction of the previous section to the multiple Eichler integrals defined in \Cref{higherper}. This allows to analytically extend $r_{f_1, \dots, f_n}$ beyond $\mathbb H^-$ and this in a way that respects the action of $\Gamma$. In the notation of \Cref{coh}, let $G=\Gamma$, $M_j=\DD_{2-k_j}$, and let $\Gamma$ act on each $M_j$ by $|_{2-k_j, \chi_j}$. We recall that, under the identification \eqref{ident}, the diagonal action of $\Gamma$ on $\otimes_{j=1}^m \DD_{2-k_j}$ is denoted by $|_{2-k_1, \dots, 2-k_m}$ (see \eqref{|m}).
With this notation, we have the following:
\begin{theorem} \label{BasThgen}
	 Let 
	$[r_{f_1, \dots, f_n}]$ be the image of
	$r_{f_1, \dots, f_n}$ under the projection $\pi_{(n)}$ in \eqref{natpr}. Then
	$$
		[r_{f_1, \dots, f_n}] \in Z^1_{(n)}\left(\Gamma; \otimes_{j=1}^n D^{\omega, \infty, \textup{exc}}_{2-k_j} \right).
	$$
\end{theorem}
\begin{proof}
We work as in Lemma 2.5 of \cite{BCD} to see that $r_{f_1, \dots, f_n}( \gamma) \in \otimes_{j=1}^n \DD_{2-k_j}$. Specifically,
$$\text{prj}_{(2-k_1)+\dots+(2-k_n)}(r_{f_1, \dots, f_n}(\gamma))(\tau)=
 \int_{-\frac dc}^{i \infty} \int_{w_1}^{i \infty} \dots \int_{w_{n-1}}^{i \infty} \prod_{j=1}^n i^{2-k_j} f_j(w_j) \left ( \frac{i-\tau}{w_j-\tau} \right )^{2-k_j} dw_n \dots dw_1, $$
where $-\frac{d}{c}=\gamma^{-1}i \infty$. This is holomorphic in $\tau$ provided that the integration paths avoid the branch cuts defined by $(\frac{i-\tau}{w_j-\tau})^{2-k_j}$. Therefore our function is holomorphic in a $\frac{d}{c}$-excised neighbourhood $\Psi$ (see Figure \ref{fig-sp}).  Hence, it belongs to $\otimes_{j=1}^n D_{2-k_j}^{\omega, \text{exc}}$. In addition, it has an asymptotic expansion at the singular point $-\frac{d}{c}$ and hence, it belongs to $\otimes_{j=1}^n D_{2-k_j}^{\omega, \infty}.$
Furthermore, by Theorem \ref{propfinn}, $r_{f_1, \dots, f_n}$ satisfies \eqref{Z1}, with $\sigma(\gamma)=r_{f_1, \dots, f_n}(\gamma)$, $\varrho_{j, s}(\gamma)=r_{f_1, \dots, f_j}( \gamma)$, and $\phi_{j, s}=\sigma_j$ (for $1\le j\le n-1$), where $s$ ranges over a singleton.
\end{proof}
We finally apply our construction to the case of integral weight. Now $G=\Gamma_0(N)$, $k_j \in 2 \mathbb N$, $M_j=\mathbb C_{k_j-2}[z]$, and $G$ acts on $M_j$ by $|_{2-k_j}$. In this case, $\otimes_{j=1}^m \mathbb C_{2-k_j}[z]$ is identified, under \eqref{ident},  with $\mathbb C_{(2-k_1)+\dots+(2-k_n)}[z]$ and the diagonal action of $|_{2-k_1, \dots, 2-k_m}$ can be identified with $|_{(2-k_1)+\dots+(2-k_m)}$. Since $r_{f_1, \dots, f_n}(\gamma) \in \mathbb C_{(2-k_1)+\dots+(2-k_n)}[z]$, Theorem \ref{propfinn} implies the following:
\begin{proposition} \label{BasThZ}
	 Let $k_j \in 2 \mathbb N$ and $f_j \in S_{k_j}(\Gamma_0(N))$. For $j\in\{1, \dots n\}$, let $[r_{f_1, \dots, f_n}]$ be the image of
	$r_{f_1, \dots, f_n}$ under the natural projection $\pi_{(n)}$ in \eqref{natpr}.
		We have $$
		[r_{f_1, \dots, f_n}] \in Z^1_{(n)}\left(\Gamma; \otimes_{j=1}^n \mathbb C_{k_j-2}[z] \right).$$
\end{proposition}
The proof of Theorem \ref{lincomb} can be readily adjusted, in accordance to the $Z^1_{(n)}$-relation satisfied by $r_{f_1, \dots, f_n}$, to give the following:
\begin{proposition}\label{lincombgen} For each $\gamma_1, \gamma_2 \in \Gamma$, there exist $(k_1-2)+\dots+ (k_n-2)+1$ 
$\mathbb Q$-
linear combinations of
$\Lambda_{f_1, \dots, f_n}((\gamma_1 \gamma_2)^{-1}i \infty, m_1, \dots m_n)$, $\Lambda_{f_1, \dots, f_n}(\gamma_1^{-1}i \infty, m_1, \dots m_n)$ and $\Lambda_{f_1, \dots f_n}(\gamma_2^{-1}i \infty, m_1, \dots m_n)$ for $1 \le m_j \le k_j-1$,
each of which equals a $\mathbb Q$-linear combination of products of the terms $\Lambda_{f_\ell, \dots, f_r}((\gamma_1 \gamma_2)^{-1}i \infty, m_\ell, \dots, m_r)$, $\Lambda_{f_\ell, \dots, f_r}(\gamma_1^{-1}i \infty, m_\ell, \dots, m_r)$ and $\Lambda_{f_\ell, \dots, f_r}(\gamma_2^{-1}i \infty, m_\ell, \dots, m_r)$, for $1 \le m_j \le k_j-1$.
\end{proposition}

\end{document}